\documentclass{article}
\usepackage{amsmath,amssymb,amsthm,thmtools,enumitem,bm,graphicx,xcolor,microtype}
\usepackage[T1]{fontenc}
\usepackage{tikz}
\usetikzlibrary{calc}

\usepackage[style = ext-numeric, sorting = nyt, giveninits = false, maxbibnames = 8, minbibnames = 3, date = year, isbn = false, doi = true, eprint = true, url = true, backref = true]{biblatex}
\renewbibmacro{in:}{}
\DeclareFieldFormat{titlecase}{\MakeSentenceCase*{#1}}
\DeclareFieldFormat{titlecase:journaltitle}{{#1}}

\usepackage[unicode]{hyperref}

\theoremstyle{definition}
\newtheorem{dfn}{Definition}[section]

\newtheorem{ex}[dfn]{Example}
\newtheorem{case}{Case}
\newtheorem{cnd}{Condition}

\theoremstyle{remark}
\newtheorem{rmk}[dfn]{Remark}

\theoremstyle{plain}
\newtheorem{thm}[dfn]{Theorem}
\newtheorem*{thm*}{Theorem}
\newtheorem{prop}[dfn]{Proposition}
\newtheorem{lem}[dfn]{Lemma}
\newtheorem{cor}[dfn]{Corollary}

\newtheorem{clm}{Claim}

\renewcommand{\qedsymbol}{$\blacksquare$}

\declaretheoremstyle[headfont = \normalfont\itshape, qed = $\blacksquare_{\mathrm{Claim}}$]{prooc}

\def\sectionautorefname~#1\null{Chapter~#1\null}
\def\subsectionautorefname~#1\null{Section~#1\null}

\title{Distality in ordered abelian groups}
\author{Koki Okura\thanks{Doctoral Program in Mathematics, University of Tsukuba, Ibaraki 305-8571, Japan\\E-mail: \href{\detokenize{mailto:k_okura@math.tsukuba.ac.jp}}{\detokenize{k_okura@math.tsukuba.ac.jp}}}}
\date{}

\begin{document}

\maketitle

\begin{abstract}
    We provide a characterization of distal ordered abelian groups:
    An ordered abelian group is distal if and only if, for each prime number $p$, the sizes of ribs with respect to the ``valuation'' $\mathfrak{s}_p$ are uniformly bounded.
    This generalizes the distality criterion for ordered abelian groups with finite spines given by Aschenbrenner, Chernikov, Gehret, and Ziegler.
    Furthermore, we prove that every ordered abelian group has a distal expansion, affirmatively answering a conjecture posed by the same authors.
\end{abstract}

\tableofcontents

\addcontentsline{toc}{section}{Introduction}
\subsection*{Introduction}

The study of NIP theories is one of the central themes in recent model theory.
This class encompasses not only all stable theories but also many important unstable theories, such as the real field, the $p$-adic fields, and algebraically closed valued fields.
Heuristically, NIP structures are often viewed as being constructed from stable parts and linear orders.
For example, an algebraically closed valued field is built from its residue field (which is algebraically closed and hence stable) and its value group (which is a divisible ordered abelian group).
In this view, stable theories represent the case where the linear order part is trivial.

Simon \cite{Simon2013DisNon} introduced \emph{distal theories} to capture the opposite: NIP theories whose stable part is trivial.
Examples of distal structures include the real field and the $p$-adic fields, whereas algebraically closed valued fields are not distal due to their nontrivial stable residue field.
Beyond their role in classification theory, distal theories possess remarkable combinatorial properties, such as strong Erd\H{o}s-Hajnal property and upper bounds for Zarankiewicz's problem \cite{Chernikov2018RegLem,Chernikov2020CutLem}, the latter of which is applicable to theories that merely have distal expansions.

In the context of valued fields, many results demonstrate that properties of henselian valued fields are controlled by those of their value groups and residue fields.
A classic example is Delon's one \cite{Delon1978TypSur} that a henselian valued field of equicharacteristic $0$ has NIP if and only if both its value group and residue field have NIP.
Such transfer results are called Ax-Kochen-Ershov (AKE) principle.
More recently, Aschenbrenner, Chernikov, Gehret, and Ziegler \cite{Aschenbrenner2022DisVal} established AKE principle on distality:
\begin{quote}
    Let $K$ be a henselian valued field with a value group $\Gamma$ and a residue field $\bm{k}$.
    Then $K$ is distal (has a distal expansion) in the language of valued fields if and only if $K$ is finitely ramified, $\Gamma$ is distal (has a distal expansion) in the language of ordered abelian group, and $\bm{k}$ is distal (has a distal expansion) in the language of fields.
\end{quote}
These AKE-style results reduce the study of henselian valued fields to the study of ordered abelian groups and pure fields, from which a great motivation for model-theoretic study of ordered abelian groups arises.

Model theory of ordered abelian groups began in the 1960s.
Following several early studies, quantifier elimination for general ordered abelian groups was established by Gurevich \cite{Gurevich1968DecPro} and Schmitt \cite{Schmitt1982ModThe}.
They subsequently used these results to prove that every ordered abelian group has NIP \cite{Gurevich1984TheOrd}.
Regrettably, their quantifier elimination results were contained in unpublished theses and remained largely overlooked for many years.
Later, Cluckers and Halupczok \cite{Cluckers2011QuaEli} independently developed quantifier elimination in a modern and simplified form using many-sorted languages, which is convenient for contemporary research.

Building on the work of Cluckers and Halupczok, Aschenbrenner et al. (ibid.) provided a partial characterization of distality for ordered abelian groups.
For an integer $n \geq 2$, an ordered abelian group $G$, and $a \in G \setminus nG$, let $\mathfrak{s}_n(a)$ be the largest convex subgroup $H$ such that $a \notin H + nG$.
If $a \in nG$, we let $\mathfrak{s}_n(a) = \emptyset$ (this convention differs from the original one by Cluckers and Halupczok).
We also define $\mathcal{S}_n = \left\{ \mathfrak{s}_n(a) \mid a \in G \right\}$.
We say that $G$ has \emph{finite spines} if $\mathcal{S}_p$ is finite for every prime number $p$.
Aschenbrenner et al. \cite[Theorem 3.13]{Aschenbrenner2022DisVal} proved that:
\begin{quote}
    An ordered abelian group $G$ with finite spines is distal if and only if $G / pG$ is finite for every prime number $p$.
\end{quote}

In this paper, we establish a criterion for distality that applies to \emph{all} ordered abelian groups, removing the assumption of finite spines.
For each $H \in \mathcal{S}_p \setminus \left\{ \emptyset \right\}$, let $G_{\leq H}^{\mathfrak{s}_p} = \left\{ a \in G \mid \mathfrak{s}_p(a) \subseteq H \right\}$ and $G_{< H}^{\mathfrak{s}_p} = \left\{ a \in G \mid \mathfrak{s}_p(a) \subsetneq H \right\}$.
These are subgroups of $G$ as we shall see later.
We then define the \emph{rib} at $H$ with respect to $\mathfrak{s}_p$, denoted by $R_{H}^{\mathfrak{s}_p}$, as the quotient $G_{\leq H}^{\mathfrak{s}_p} / G_{< H}^{\mathfrak{s}_p}$.
Our first main result is the following:

\begin{thm*}
    An ordered abelian group $G$ is distal if and only if for each prime number $p$, there exists a positive integer $N_p$ such that for every $H \in \mathcal{S}_p \setminus \left\{ \emptyset \right\}$, $|R_{H}^{\mathfrak{s}_p}| < N_p$.
\end{thm*}

In other words, $G$ is distal if and only if for every prime number $p$, the sizes of ribs with respect to $\mathfrak{s}_p$ are uniformly bounded.
For any $H \in \mathcal{S}_p \setminus \left\{ \emptyset \right\}$, the rib $R_{H}^{\mathfrak{s}_p}$ can be viewed as an $\mathbb{F}_p$-vector space, as every element is annihilated by $p$.
If the sizes of ribs are not uniformly bounded for some $p$, then there exists $H \in \mathcal{S}_p \setminus \left\{ \emptyset \right\}$ in an elementary extension of $G$ such that $|R_{H}^{\mathfrak{s}_p}| = \infty$.
Our theorem can be interpreted as stating that the existence of such an infinite $\mathbb{F}_p$-vector space induces stable behavior, which is the sole source of non-distality in an ordered abelian group.A

Note that for any $H_1 \subsetneq H_2$ in $\mathcal{S}_p \setminus \left\{ \emptyset \right\}$, we have $pG \subseteq G_{< H_1}^{\mathfrak{s}_p} \subsetneq G_{\leq H_1}^{\mathfrak{s}_p} \subseteq G_{< H_2}^{\mathfrak{s}_p} \subsetneq G_{\leq H_2}^{\mathfrak{s}_p} \subseteq G$.
Thus, if $\mathcal{S}_p$ is infinite, then so is $G / pG$.
Conversely, suppose that $\mathcal{S}_p$ is finite and listed as $\emptyset \subsetneq H_1 \subsetneq \dots \subsetneq H_m$.
In this case, $pG = G_{< H_1}^{\mathfrak{s}_p}$, $G_{\leq H_i}^{\mathfrak{s}_p} = G_{< H_{i+1}}^{\mathfrak{s}_p}$ for every $i < m$, and $G_{\leq H_m}^{\mathfrak{s}_p} = G$.
Consequently, we have $|G / pG| = \prod_{i=1}^m |R_{H_i}^{\mathfrak{s}_p}|$.
Therefore, the partial characterization by Aschenbrenner et al. is a special case of our general result.

As for the existence of distal expansions, Aschenbrenner et al. \cite[Conjecture 3.16]{Aschenbrenner2022DisVal} conjectured that every ordered abelian group admits a distal expansion.
We affirmatively answer this conjecture.

\begin{thm*}
    Every ordered abelian group has a distal expansion.
\end{thm*}

In addition to $(\mathfrak{s}_n)_{n\geq 2}$, another family of map $(\mathfrak{t}_n)_{n\geq 2}$ is defined.
These two kinds of maps serve as the key tools for our results.
An immediate but important observation is that these maps can be viewed as valuations.
Furthermore, the quantifier elimination by Cluckers and Halupczok \cite{Cluckers2011QuaEli} reduces the type of an element $a$ to those of its values $(\mathfrak{s}_{p^r}(a))_{\text{$p$:prime,}\,r \geq 1}$ and $(\mathfrak{t}_p(a))_{\text{$p$:prime}}$.
Our first result asserts that the presence of ribs of unbounded sizes with respect to $\mathfrak{s}_p$ destroys distality.
The core idea for obtaining a distal expansion of an ordered abelian group is to refine these valuations so that the resulting ribs are of size $p$.

Combining our results with that of Aschenbrenner et al., the problem of determining which henselian valued fields are distal (have distal expansions) reduces to determining which pure fields are distal (have distal expansions).
Indeed, they provided a conjectural classification of henselian valued fields admitting distal expansions.
Shelah's NIP fields conjecture states that every infinite NIP field is either separably closed, real closed, or admits a nontrivial henselian valuation.
Aschenbrenner et al. \cite[Corollary 6.15]{Aschenbrenner2022DisVal} showed the following:
\begin{quote}
    Assume that NIP fields conjecture holds, and that every ordered abelian group admits a distal expansion.
    If a field $K$ has NIP, then the following are equivalent:
    \begin{itemize}
        \item $K$ has a distal expansion;
        \item $K$ does not interpret an infinite field of characteristic $p > 0$;
        \item $K$ does not define a valuation ring whose residue field is infinite of characteristic $p > 0$;
        \item $K$ has a henselian valuation ring whose residue field is either algebraically closed of characteristic $0$, real closed, or finite.
    \end{itemize}
\end{quote}
Since we establish the latter assumption, the NIP fields conjecture is the only remaining problem for this classification.

Let us discuss another aspect of our results.
As mentioned earlier, many model-theoretic properties of henselian valued fields are derived from those of their value groups and residue fields, and such transfer results are referred to as AKE principle.
On the other hand, the valuations $\mathfrak{s}_n$ and $\mathfrak{t}_n$ are essential for describing the first-order theory of ordered abelian groups.
Since spines and ribs are the counterparts of value groups and residue fields, one expects that their model-theoretic properties to be inherited by the entire ordered abelian groups, yielding AKE principle for ordered abelian groups.
Indeed, several previous results can be viewed from this perspective:

\begin{quote}
    If we introduce to spines unary predicates indicating the architecture of ribs, then the theory of $G$ is determined by that of the spines (\cite[Corollary 4.7]{Schmitt1982ModThe}).
\end{quote}

\begin{quote}
    The dp-rank of $G$ is $1 + \sum_{p\in\mathbb{P}} \left| \left\{ H \in \mathcal{S}_p \setminus \left\{ \emptyset \right\} \;\middle|\; |R_H^{\mathfrak{s}_p}| = \infty \right\} \right|$ if $G$ has finite spines (\cite[Theorem 7.2]{Farre2017StrOrd}).
\end{quote}

\begin{quote}
    Assume that $\mathfrak{t}_{\infty}$, defined by $\mathfrak{t}_{\infty}(a) = \bigcap_{\text{$p$:prime}} \mathfrak{t}_p(a)$, is definable and that $(G,\mathfrak{t}_{\infty})$ is maximal as a valued abelian group.
    If the ribs and spine with respect to $\mathfrak{t}_{\infty}$ are stably embedded, then $G$ itself is stably embedded \cite[Theorem 2.28]{Hils2023StaEmb}.
\end{quote}

Our results on distality can also be understood as an instance of this AKE principle.
While the spines are colored orders (i.e., linear orders with unary predicates) and hence distal, the ribs with respect to $\mathfrak{s}_p$ are $\mathbb{F}_p$-vector spaces and thus stable; consequently, they are distal if and only if they are finite.
On the other hand, $\mathbb{F}_p$-vector spaces always have a distal expansion:
For $V = \bigoplus_{i \in I} \mathbb{F}_p$, introduce a linear order on $I$ and define a valuation $v: V \to I \cup \left\{ -\infty \right\}$ by $v(0) = -\infty$ and $v\left( (a_i)_{i \in I} \right) = \max \left\{ i \in I \mid a_i \neq 0 \right\}$.
One can show that $(V,+,v)$ is distal (we omit the proof).
Therefore, our results state that the distality, or the existence of a distal expansion, of the ribs and spines is inherited by the whole group.

The structure of this paper is as follows.
In \autoref{chp:Notation}, we establish the notation and review some basic facts regarding distality.
\autoref{chp:Review} collects the essential facts required to prove distality and non-distality.
This chapter consists of a detailed study of the valuations $\mathfrak{s}_n$ and $\mathfrak{t}_n$, illustrative examples of ordered abelian groups, and a brief review of the relative quantifier elimination by Cluckers and Halupczok \cite{Cluckers2011QuaEli}.
Most of the definitions and statements are taken from that paper.
The characterization of distality for ordered abelian groups is given in \autoref{chp:Distality}.
Both directions, distality under bounded rib sizes and non-distality under unbounded rib sizes, require some effort and are each treated in a dedicated section.
We show the existence of a distal expansion in \autoref{chp:DisExp}.
We construct valuations refining $(\mathfrak{s}_{p^r})_{\text{$p$:prime,}\,r \geq 1}$, establish relative quantifier elimination for an expanded language adapting the framework of Cluckers and Halupczok, and then prove distality.

\section*{Acknowledgement}

This achievement was supported by JST SPRING, Grant Number JPMJSP2124.

\section{Notation and preliminaries}
\label{chp:Notation}

\subsection{Notation}
\label{sec:Notation}

We use $\overline{a}$, $\overline{b}$,... to denote tuples of finite length unless otherwise specified (we sometimes deal with tuples of infinite length), and $|\overline{a}|$ denotes the length of the tuple $\overline{a}$.

For linearly ordered sets $I$ and $J$, we denote by $I + J$ the disjoint union $I \sqcup J$ with the order such that $i < j$ for any $i \in I$ and $j \in J$.
We write $(c)$ to refer to a singleton $\left\{ c \right\}$ with its unique linear order.
Thus, $I + (c) + J$ is the linearly ordered set in which $c$ is placed between $I$ and $J$.
For $i \in I$, we define $I_{\square i} = \left\{ j \in I \mid j \mathrel{\square} i \right\}$, where $\square$ is either $>$, $\geq$, $<$, or $\leq$.
For $S, T \subseteq I$, we write $S < T$ if every element of $S$ is less than every element of $T$.
We adopt the same notation $\overline{a} < \overline{b}$ for tuples as well.
For a sequence $(a_i)_{i \in I}$ and a tuple of indices $\bar{\imath} = (i_1,\dots,i_n)$, $a_{\bar{\imath}}$ denotes the tuple $(a_{i_1},\dots,a_{i_n})$.

We let $\mathbb{P}$ denote the set of prime numbers.
For each $p \in \mathbb{P}$, $v_p$ denotes the $p$-adic valuation; that is, for a nonzero integer $m$, $v_p(m)$ is the maximal integer $r$ such that $p^r \mid m$.

Throughout this paper, $G$ denotes an ordered abelian group $(G;<,+)$ unless otherwise specified.
We regard $G$ as a structure in the language $L_{\mathrm{oag}} = \left\{ <,+ \right\}$ or $L_{\mathrm{syn}}$, the latter of which is defined in \autoref{sec:ReviewSyn}.
$G$ is called \emph{discrete} if it has a smallest positive element, and \emph{dense} otherwise.
We say $G$ is \emph{archimedean} if for all positive elements $a$ and $b$, there exists a positive integer $n$ such that $na > b$.

We follow standard notation for groups: For instance, $nG = \left\{ na \mid a \in G \right\}$ for an integer $n$.
For two subsets $A$ and $B$ of $G$, $A + B$ denotes $\left\{ a + b \mid a \in A, b \in B \right\}$.
For a subgroup $H$ and $a \in G$, $a / H$ denotes the image of $a$ in the quotient group $G / H$.
For $a,b \in G$, we write $a \equiv_m b$ if $a - b \in mG$.
We use the dot product notation: For a tuple of integers $\overline{m} = (m_1,\dots,m_n)$ and a tuple of elements $\overline{a} = (a_1,\dots,a_n)$ from $G$, $\overline{m} \cdot \overline{a} = \sum_{i=1}^n m_i a_i$.

We write $H \Subset G$ to indicate that $H$ is a convex subgroup of $G$.
For two convex subgroups $H_1$ and $H_2$, we write $H_1 \leq H_2$ if $H_1 \subseteq H_2$.
This ordering defines a linear order on the set of convex subgroups of $G$.
We also define $\emptyset < H$ for any $H \Subset G$.
$\langle a \rangle^{\mathrm{conv}} = \left\{ x \in G \mid -k|a| \leq x \leq k|a| \text{ for some }k \in \mathbb{Z}_{>0} \right\}$ is the convex subgroup generated by $a \in G$.
If $H$ is convex, the order on $G$ induces an order on $G / H$ compatible with the addition, defined as follows:
$x / H < y / H$ if and only if $x < y$ and $y - x \notin H$.
If $G / H$ is discrete, then for each $k \in \mathbb{Z}$, we let $k_H$ denote $k$ times the smallest positive element of $G / H$.
\subsection{Generalities on distality}
\label{sec:GeneralitiesOnDistality}

By a ``monster model'', we mean a $\kappa$-big model in the sense of Hodges \cite[Chapter 10]{Hodges1993ModThe} for a sufficiently large cardinal $\kappa$.
It is irrelevant what $\kappa$-bigness actually is.
What is important is the following facts (within ZFC):
Every structure has a $\kappa$-big elementary extension for any $\kappa$;
$\kappa$-bigness implies $\kappa$-saturation and strong $\kappa$-homogeneity;
interpretations preserve $\kappa$-bigness.
For these reasons, $\kappa$-big models are suitable as workspaces in which we examine various properties of a theory.

In this section, $T$ denotes a complete, possibly many-sorted theory.
All elements are assumed to be taken from a monster model of $T$.
Let $DLO$ denote the theory of dense linear orders without endpoints.

\begin{dfn}
    We say $T$ is \emph{distal} if for every $I, J \models DLO$, every sequence $(\overline{a}_k)_{k \in I+(c)+J}$, and every $\overline{b}$ (where $|\overline{a}_k|$ and $|\overline{b}|$ are possibly infinite), the following holds: If
    \begin{itemize}
        \item $(\overline{a}_k)_{k \in I+(c)+J}$ is indiscernible, and
        \item $(\overline{a}_k)_{k \in I+J}$ is indiscernible over $\overline{b}$,
    \end{itemize}
    then $(\overline{a}_k)_{k \in I+(c)+J}$ is indiscernible over $\overline{b}$.
    We say a structure is distal (has a distal expansion) if its theory is distal (has a distal expansion).
\end{dfn}

Simon \cite{Simon2013DisNon} established many fundamental properties of distality upon its introduction.
Results not treated in that paper are covered in \cite[Chapter 9]{Simon2015GuiNIP} and \cite[Chapter 1]{Aschenbrenner2022DisVal}.
Furthermore, distal theories are known to have NIP (\cite[Proposition 2.9]{Gehret2020DisAsy}).
We cite two results for later use.
The first provides a sufficient condition for distality that appears weaker than the original definition.

\begin{prop}[A slightly adapted version of {\cite[Proposition 1.10 and Corollary 1.11]{Aschenbrenner2022DisVal}}]
    Fix arbitrary $I, J \models DLO$.
    Suppose that for every sequence $(\overline{a}_k)_{k \in I+(c)+J}$ and every $\overline{b}$, if:
    \begin{itemize}
        \item $(\overline{a}_k)_{k \in I+(c)+J}$ is indiscernible, and
        \item $(\overline{a}_k)_{k \in I+J}$ is indiscernible over $\overline{b}$,
    \end{itemize}
    then $\mathrm{tp}(\overline{a}_c / \overline{b}) = \mathrm{tp}(\overline{a}_k / \overline{b})$ for any $k \in I + J$.
    Then $T$ is distal.
    
    Moreover, the same statement holds with either of the following additional conditions:
    \begin{itemize}
        \item $|\overline{a}_k| = 1$ and $|\overline{b}|$ is finite.
        \item $|\overline{a}_k|$ is finite and $|\overline{b}| = 1$.
    \end{itemize}
\end{prop}

\begin{prop}[cf. {\cite[Corollary 1.27]{Aschenbrenner2022DisVal}}]
    Suppose that $T$ is one-sorted.
    Let $\widetilde{T}$ be a complete theory obtained from $T$ by adding imaginary sorts to the language of $T$.
    Then $T$ is distal if and only if $\widetilde{T}$ is distal.
\end{prop}

Consequently, to prove the distality of a theory with imaginary sorts, it suffices to consider sequences $(\overline{a}_i)_{i \in I+(c)+J}$ and $\overline{b}$ both from the real sort.

\section{Review on ordered abelian groups}
\label{chp:Review}

\subsection{Valuations}
\label{sec:ReviewValuations}

In this section, we introduce three functions: $\mathfrak{s}_n$, $\mathfrak{t}_n$, and $\mathfrak{t}_n^+$.
As we will see, these functions essentially behave as valuations.\footnote{Despite its importance, the valuation-like behavior of these functions has received surprisingly little attention: We are aware of only one paper (\cite{Hils2023StaEmb}), that mentions this.}
In our argument concerning distality, the first two ``valuations'' play a major role, while $\mathfrak{t}_n^+$ is essentially equivalent to $\mathfrak{t}_n$.
We therefore examine their properties closely in this section.
Although many lemmas here are taken from Cluckers and Halupczok \cite{Cluckers2011QuaEli}, we include their proofs for completeness as they require little effort.

We begin with some basic facts regarding abelian groups.

\begin{lem}\label{lem:CRT}
    Let $G$ be a (not necessarily ordered) abelian group, $a, b \in G$, and $m,n$ positive coprime integers.
    \begin{enumerate}
        \item $mG \cap nG = mnG$.
        \item $G = mG + nG$. Moreover, $(a + mG) \cap (b + nG)$ is a coset of $mnG$.
        \item If $ma \in nG$, then $a \in nG$.
    \end{enumerate}
\end{lem}

\begin{proof}
    Since $m$ and $n$ are coprime, there exist integers $s$ and $t$ such that $1 = sm + tn$.
    This implies (1), $G = mG + nG$, and (3).
    Since $a - b \in mG + nG$, the intersection $(a + mG) \cap (b + nG)$ is not empty.
    Hence, it is a coset of $mG \cap nG = mnG$.
\end{proof}

We fix an ordered abelian group $G$ from now on.
Keeping in mind that the union and intersection of convex subgroups are again convex subgroups, we give the following definition.

\begin{dfn}[{\cite[Definition 1.1]{Cluckers2011QuaEli}}]
    Fix an integer $n \geq 2$.
    \begin{itemize}
        \item For $a \in G \setminus nG$, $\mathfrak{s}_n(a)$ is the largest convex subgroup $H \Subset G$ such that $a \notin H + nG$.
        If $a \in nG$, we set $\mathfrak{s}_n(a) = \emptyset$.
        We define $\mathcal{S}_n = \left\{ \mathfrak{s}_n(a) \mid a \in G \right\}$.\footnote{In \cite{Cluckers2011QuaEli}, $\mathfrak{s}_n(a)$ is denoted by $H_a$, and the symbol $\mathfrak{s}_n(a)$ itself refers to the equivalence class defined by the relation $H_a = H_b$.
        Here, we adopt a more direct notation.
        Additionally, in that paper, $\mathfrak{s}_n(a)$ is defined to be $\left\{ 0 \right\}$ for $a \in nG$. We find our current convention more convenient for treating $\mathfrak{s}_n$ as a valuation.}
        \item For $a \in G$, $\mathfrak{t}_n(a) = \bigcup_{H \in \mathcal{S}_n, a \notin H} H$.
        Also, $\mathfrak{t}_n^+(a) = \bigcap_{H \in \mathcal{S}_n, a \in H} H$, where the empty intersection is defined as $G$.
        We define $\mathcal{T}_n = \left\{ \mathfrak{t}_n(a) \mid a \in G \right\}$ and $\mathcal{T}_n^+ = \left\{ \mathfrak{t}_n^+(a) \mid a \in G \right\}$.
    \end{itemize}
\end{dfn}

Although we cite these definitions from Cluckers and Halupczok \cite{Cluckers2011QuaEli}, these maps already appeared in Schmitt \cite{Schmitt1982ModThe}:
The map $F_n$ given there is precisely $\mathfrak{s}_n$, and while the maps $A_n$ and $B_n$ were defined in terms of $n$-regularity (\autoref{dfn:nRegurality}), one can show that $A_n(g) = \mathfrak{t}_n(g)$ and $B_n(g) = \mathfrak{t}_n^+(g)$ except when $\mathfrak{t}_n(g) = \emptyset$.
Schmitt (ibid.) also refers to the union $\mathcal{S}_n \cup \mathcal{T}_n$ as the \emph{$n$-spine} and denotes it by $\mathrm{Sp}_n$.
It is immediate that if $\mathcal{S}_n$ is finite, then $\mathcal{T}_n$, $\mathcal{T}_n^+$, and consequently $\mathrm{Sp}_n$ are also finite.
As we shall see in \autoref{lem:DifficultValuationProperties}, if $\mathcal{S}_p$ is finite for all $p \in \mathbb{P}$, then $\mathcal{S}_n$ is finite for all $n \geq 2$.
Accordingly, we say that $G$ has \emph{finite spines} if $|\mathcal{S}_p| < \infty$ for all $p \in \mathbb{P}$.

\begin{lem}[{\cite[Lemma 2.1]{Cluckers2011QuaEli}}]
    For each $n \geq 2$, $\mathfrak{s}_n(a)$ is definable uniformly in $G$ and $a$.
    That is, there exists a formula $\varphi(x,y)$ in $L_{\mathrm{oag}} = \left\{ <,+ \right\}$ such that for any ordered abelian group $G$ and any $a \in G$, we have $\varphi(G,a) = \mathfrak{s}_n(a)$.
    Consequently, $\mathfrak{t}_n(a)$ and $\mathfrak{t}_n^+(a)$ are also uniformly definable.
\end{lem}

\begin{proof}
    We have $b \in \mathfrak{s}_n(a) \iff \langle b \rangle^{\mathrm{conv}} \leq \mathfrak{s}_n(a) \iff a \notin \langle b \rangle^{\mathrm{conv}} + nG$.
    Since $\langle b \rangle^{\mathrm{conv}} + nG = [0,n|b|] + nG$, the uniform definability of $\mathfrak{s}_n$ is ensured, which in turn implies the uniform definability of $\mathfrak{t}_n$ and $\mathfrak{t}_n^+$.
\end{proof}

Recall that the convex subgroups are ordered by inclusion. $\mathfrak{t}_n$ and $\mathfrak{t}_n^+$ are equivalent in the following sense:

\begin{lem}\label{lem:tArrangement}
    For $a,b \in G$, $\mathfrak{t}_n(a) = \mathfrak{t}_n(b) \implies \mathfrak{t}_n^+(a) = \mathfrak{t}_n^+(b)$ and $\mathfrak{t}_n(a) < \mathfrak{t}_n(b) \implies \mathfrak{t}_n^+(a) \leq \mathfrak{t}_n(b)$.
    In particular, the map $\mathfrak{t}_n(a) \mapsto \mathfrak{t}_n^+(a)$ defines an order isomorphism between $\mathcal{T}_n$ and $\mathcal{T}_n^+$.
\end{lem}
The proof is straightforward.
By virtue of this lemma, the properties of $\mathfrak{t}_n$ established below are inherited by $\mathfrak{t}_n^+$.

A \emph{valued abelian group} is a triple $(G,S,v)$, where $G$ is an abelian group, $S$ is a linearly ordered set, and $v:G \to S \cup \left\{ \infty \right\}$ is a surjective map satisfying for all $x, y \in G$:
\begin{itemize}
    \item $v(x) = \infty \iff x = 0$.
    \item $v(-x) = v(x)$.
    \item $v(x + y) \geq \min \left\{ v(x), v(y) \right\}$.
\end{itemize}
Such a map $v$ is called a \emph{valuation} on $G$.
See \cite[Section 2.2]{Aschenbrenner2017AsyDif} for a reference to general theory of valued abelian groups.
The next lemma shows that $\mathfrak{s}_n$ and $\mathfrak{t}_n$ behave similarly to valuations.

\begin{lem}\label{lem:ValuationAxioms}
    Let $n \geq 2$ and $a,b \in G$.
    \begin{enumerate}
        \item $\mathfrak{s}_n(-a) = \mathfrak{s}_n(a)$ and $\mathfrak{t}_n(-a) = \mathfrak{t}_n(a)$.
        \item $\mathfrak{s}_n(a + b) \leq \max \left\{ \mathfrak{s}_n(a), \mathfrak{s}_n(b) \right\}$ and $\mathfrak{t}_n(a + b) \leq \max \left\{ \mathfrak{t}_n(a), \mathfrak{t}_n(b) \right\}$.
    \end{enumerate}
\end{lem}

\begin{proof}
    (1) is immediate from the definitions.
    For the first part of (2), suppose that $H \Subset G$ is larger than both $\mathfrak{s}_n(a)$ and $\mathfrak{s}_n(b)$.
    Then, $a, b \in H + nG$, which is a subgroup.
    Thus $a + b \in H + nG$, implying $H > \mathfrak{s}_n(a + b)$.
    The proof for $\mathfrak{t}_n$ is analogous.
\end{proof}

To regard $\mathfrak{s}_n$ as a standard valuation, one would need to reverse the ordering of $\mathcal{S}_n$.
Furthermore, since $\mathfrak{s}_n$ assigns $\emptyset$ to all elements in $nG$, it is more accurately a valuation on the quotient $G / nG$.
Nonetheless, we find these properties sufficiently useful to refer to $\mathfrak{s}_n$ and $\mathfrak{t}_n$ as valuations.

\begin{rmk}\label{rmk:TriangleIsIsosceles}
    As with standard valuations, if $\mathfrak{s}_n(a) \neq \mathfrak{s}_n(b)$, then $\mathfrak{s}_n(a + b) = \max \left\{ \mathfrak{s}_n(a), \mathfrak{s}_n(b) \right\}$.
    Equivalently, for any $a,b,c \in G$, at least two of $\mathfrak{s}_n(a - b)$, $\mathfrak{s}_n(b - c)$, and $\mathfrak{s}_n(c - a)$ are equal, and the third is less than or equal to the others.
    In other words, any ``triangle'' with respect to $\mathfrak{s}_n$ is isosceles, with its ``legs'' being greater than or equal to the ``base''.
    This property also holds for $\mathfrak{t}_n$.
\end{rmk}

We keep examining properties of the valuations.

\begin{rmk}\label{rmk:ConvexityConsequence}
    For a nonzero integer $k$, $a \in G$, and $H \Subset G$, we have $ka \in H \iff a \in H$.
    Consequently, $a \in H + nG \iff ka \in H + knG$ for any $n \geq 2$.
\end{rmk}

\begin{lem}\label{lem:EasyValuationProperties}
    Let $m,n \geq 2$ and $a \in G$.

    \begin{enumerate}
        \item $\mathfrak{s}_n(a) \leq \mathfrak{s}_{mn}(a)$.
        \item $\mathfrak{s}_n(a) = \mathfrak{s}_{mn}(ma)$.
        \item If $m$ and $n$ are coprime, then $\mathfrak{s}_n(ma) = \mathfrak{s}_n(a)$.
        \item $\mathfrak{t}_n(ma) = \mathfrak{t}_n(a)$.
    \end{enumerate}
\end{lem}

\begin{proof}
    (1) follows immediately from the definition.
    (2) and (4) are consequences of \autoref{rmk:ConvexityConsequence}.
    For (3), apply \autoref{lem:CRT} (3) to the quotient group $G / H$ for any $H \Subset G$.
\end{proof}

\begin{lem}[{\cite[Lemma 2.2]{Cluckers2011QuaEli}}]\label{lem:DifficultValuationProperties}
    Let $a \in G$, $n \geq 2$, $p \in \mathbb{P}$, and $r \geq 1$.

    \begin{enumerate}
        \item There exists $a' \in G$ such that $\mathfrak{s}_p(a') = \mathfrak{s}_{p^r}(a)$.
        \item $\mathfrak{s}_n(a) = \max_{p \in \mathbb{P}, p \mid n} \left\{ \mathfrak{s}_{p^{v_p(n)}}(a) \right\}$.
        \item $\mathcal{S}_n = \bigcup_{p \in \mathbb{P}, p \mid n} \mathcal{S}_p$ for any $n \geq 2$. In particular, $\mathcal{S}_{p^r} = \mathcal{S}_p$.
        \item $\mathfrak{t}_n(a) = \max_{p \in \mathbb{P}, p \mid n} \left\{ \mathfrak{t}_p(a) \right\}$ and $\mathfrak{t}_n^+(a) = \min_{p \in \mathbb{P}, p \mid n} \left\{ \mathfrak{t}_p^+(a) \right\}$.
    \end{enumerate}
\end{lem}

\begin{proof}
    (1) Let $s$ be the maximal integer such that $0 \leq s < r$ and $a \in \mathfrak{s}_{p^r}(a) + p^s G$.
    Then $a$ can be expressed as $a = b + p^s a'$ for some $b \in \mathfrak{s}_{p^r}(a)$ and $a' \in G$.
    It follows that $a' \notin \mathfrak{s}_{p^r}(a) + p G$; otherwise, $a = b + p^s a' \in \mathfrak{s}_{p^r}(a) + p^{s+1} G$, which contradicts the maximality of $s$.
    On the other hand, for any $H \Subset G$ greater than $\mathfrak{s}_{p^r}(a)$, we have $b + p^s a' = a \in H + p^{s+1} G$.
    By \autoref{rmk:ConvexityConsequence}, this implies $a' \in H + p G$.
    Therefore, $\mathfrak{s}_p(a') = \mathfrak{s}_{p^r}(a)$.

    (2) By \autoref{lem:EasyValuationProperties} (1), it suffices to show $\mathfrak{s}_n(a) \leq \max_{p \in \mathbb{P}, p \mid n} \left\{ \mathfrak{s}_{p^{v_p(n)}}(a) \right\}$.
    Suppose $H \Subset G$ is greater than $\max_{p \in \mathbb{P}, p \mid n} \left\{ \mathfrak{s}_{p^{v_p(n)}}(a) \right\}$.
    Then $a \in \bigcap_{p \in \mathbb{P}, p \mid n} (H + p^{v_p(n)} G)$, which equals $H + nG$ by applying \autoref{lem:CRT} (1) to $G / H$.
    Hence, $H > \mathfrak{s}_n(a)$.

    (3) It follows from \autoref{lem:EasyValuationProperties}, (1), and (2).

    (4) It is immediate from (3).
\end{proof}

We now define \emph{ribs}, which are used to characterize the distality of ordered abelian groups.

\begin{dfn}[{\cite[Definition 1.26]{Hils2023StaEmb}}]\label{dfn:Ribs}
    For $H \in \mathcal{S}_n \setminus \left\{ \emptyset \right\}$, we define $G_{\leq H}^{\mathfrak{s}_n} = \left\{ a \in G \mid \mathfrak{s}_n(a) \leq H \right\}$ and $G_{< H}^{\mathfrak{s}_n} = \left\{ a \in G \mid \mathfrak{s}_n(a) < H \right\}$.
    These are subgroups of $G$ by \autoref{lem:ValuationAxioms}.
    We define the \emph{rib} at $H$ with respect to $\mathfrak{s}_n$ as the quotient group $R_{H}^{\mathfrak{s}_n} = G_{\leq H}^{\mathfrak{s}_n} / G_{< H}^{\mathfrak{s}_n}$.
    These groups are also defined analogously for $\mathfrak{t}_n$.
\end{dfn}

As explained in the introduction, if $|\mathcal{S}_n| = \infty$, then $|G / nG| = \infty$.
On the other hand, if $|\mathcal{S}_n| < \infty$, then $|G / nG| = \prod_{H \in \mathcal{S}_n \setminus \left\{ \emptyset \right\}} |R_{H}^{\mathfrak{s}_n}|$.

\begin{dfn}[{\cite[Definition 1.4]{Cluckers2011QuaEli}}]\label{dfn:SqBracket}
    For $H \Subset G$ and $m \geq 2$, we define $H^{[m]} = \bigcap_{H' \Subset G, H' > H} (H' + mG)$, which is a subgroup of $G$.
\end{dfn}

\begin{rmk}\label{rmk:ConsequenceOfConvexityConsequence}
    For $a \in G$, $H \Subset G$, and $m, n \geq 2$, it follows from \autoref{rmk:ConvexityConsequence} that $a \in H^{[m]} \iff na \in H^{[nm]}$.
\end{rmk}

\begin{lem}[{\cite[Lemma 2.4]{Cluckers2011QuaEli}}]
    In the setting of \autoref{dfn:SqBracket}, we have $H^{[m]} = \bigcap_{H' \in \mathcal{S}_m, H' > H} (H' + mG)$.
    In particular, for any $n \geq 2$, $\mathfrak{s}_n(a)^{[m]}$ is definable in $L_{\mathrm{oag}}$ uniformly in $G$ and $a$.
\end{lem}

\begin{proof}
    If $b \notin H^{[m]}$, there exists $H' \Subset G$ such that $H' > H$ and $b \notin H' + mG$.
    Then $\mathfrak{s}_m(b) \geq H' > H$.
    Thus, we have $b \notin \mathfrak{s}_m(b) + mG \supseteq \bigcap_{H' \in \mathcal{S}_m, H' > H} (H' + mG)$.
\end{proof}

\begin{lem}[{\cite[Lemma 2.8]{Cluckers2011QuaEli}}]\label{lem:CoverAndValue}
    \begin{enumerate}
        \item For $H \in \mathcal{S}_n \setminus \left\{ \emptyset \right\}$, $G_{\leq H}^{\mathfrak{s}_n} = H^{[n]}$ and $G_{< H}^{\mathfrak{s}_n} = H + nG$.
        \item For $H \in \mathcal{T}_n \setminus \left\{ \emptyset \right\}$, $G_{\leq H}^{\mathfrak{t}_n} = H^+$ and $G_{< H}^{\mathfrak{t}_n} = H$, where $H^+ = \mathfrak{t}_n^+(a)$ for an arbitrary $a$ such that $H = \mathfrak{t}_n(a)$.
    \end{enumerate}
\end{lem}

The proof is straightforward.
Note that any element in $R_H^{\mathfrak{s}_n} = H^{[n]} / (H + nG)$ vanishes when multiplied by $n$.
In other words, the rib $R_H^{\mathfrak{s}_n}$ can be viewed as a $(\mathbb{Z} / n \mathbb{Z})$-module and thus admits no linear ordering compatible with the group operation.
In contrast, the rib $R_{H}^{\mathfrak{t}_n} = H^+ / H$ is an ordered abelian group due to the convexity of $H$.
This is considerably simpler than the whole $G$:

\begin{dfn}\label{dfn:nRegurality}
    Let $n \geq 2$.
    An ordered abelian group $G$ is \emph{$n$-regular} if the following equivalent conditions hold:
    \begin{enumerate}
        \item Any convex subset containing at least $n$ elements intersects $nG$.
        \item Any convex subset containing at least $n$ elements intersects $nG + a$ for any $a \in G$.
        \item For any $H \Subset G$ larger than $\left\{ 0 \right\}$, the quotient $G / H$ is $n$-divisible.
    \end{enumerate}
\end{dfn}

\begin{proof}[Proof of the equivalence]
    Assume (1).
    For any $a$ and any convex subset $I$ containing $n$ elements, we have $(I - a) \cap nG \neq \emptyset$.
    This establishes (2).
    It is clear that (2) implies (3).
    To show that (3) implies (1), assume (3) and let $I$ be a convex subset containing $n$ elements.
    First, suppose that $G$ is discrete.
    In this case, we may assume $I = \left\{ a,a+1,\dots,a+n-1 \right\}$.
    (3) ensures that there exists $x \in (a + \langle 1 \rangle^{\mathrm{conv}}) \cap nG$.
    Then there exists $z \in \mathbb{Z}$ such that $x \in [a + zn, a + (z+1)n)$, and hence $x - zn \in I \cap nG$.
    Next, suppose that $G$ is dense.
    Choose $a < b$ from $I$.
    By density, we can find $c > 0$ such that $nc < b - a$.
    As in the discrete case, there exists $z \in \mathbb{Z}$ and $x \in [a + znc, a + (z+1)nc) \cap nG$.
    Then $x - znc \in I \cap nG$.
\end{proof}

\begin{lem}\label{lem:nRegurality}
    For every $H \in \mathcal{T}_n \setminus \left\{ \emptyset \right\}$, the rib $H^+ / H$ is $n$-regular.
\end{lem}

\begin{proof}
    Let $H' \Subset H^+$ be larger than $H$ and $a \in H^+ \setminus H$.
    Since $\mathfrak{s}_n(a) \leq \mathfrak{t}_n(a) = H < H'$, we have $a \in H' + nG$.
    This shows the $n$-divisibility of $H^+ / H'$.
\end{proof}

For $p \in \mathbb{P}$ and $H \in \mathcal{S}_p \setminus \left\{ \emptyset \right\}$, we have a chain of subgroups:
\begin{equation*}
    H + pG \subseteq \dots \subseteq H^{[p^3]} + pG \subseteq H^{[p^2]} + pG \subseteq H^{[p]} + pG = H^{[p]}.
\end{equation*}

\begin{lem}\label{lem:LinearlyIndependentElements}
    Let $p \in \mathbb{P}$, $H \in \mathcal{S}_p \setminus \left\{ \emptyset \right\}$, and $r \geq s \geq 1$.
    \begin{enumerate}
        \item For any $\overline{a} = (a_1,\dots,a_k)$ from $H^{[p^r]}$, the following are equivalent:
        \begin{enumerate}
            \item[i.] Their images in $(H^{[p^r]} + pG) / (H + pG)$ are $\mathbb{F}_p$-linearly independent.
            \item[ii.] For any integers $\overline{m} = (m_1,\dots,m_k)$, if $\min \left\{ v_p(m_1),\dots,v_p(m_k) \right\} = 0$, then $\mathfrak{s}_p(\overline{m} \cdot \overline{a}) = H$ ($\overline{m} \cdot \overline{a}$ is the dot product).
            \item[iii.] Their images in $R_H^{\mathfrak{s}_{p^r}} = H^{[p^r]} / (H + p^r G)$ are $(\mathbb{Z} / p^r \mathbb{Z})$-linearly independent.
            \item[iv.] For any integers $\overline{m} = (m_1,\dots,m_k)$, if $\min \left\{ v_p(m_1),\dots,v_p(m_k) \right\} < r$, then $\mathfrak{s}_{p^r}(\overline{m} \cdot \overline{a}) = H$.
        \end{enumerate}
        \item The dimension $\mathrm{dim}_{\mathbb{F}_p}\left( (H^{[p^r]} + pG) / (H + pG) \right)$ equals the maximal cardinality of a $(\mathbb{Z} / p^r \mathbb{Z})$-linearly independent subset of $R_H^{\mathfrak{s}_{p^r}}$.
        \item For any $a_1,\dots,a_k \in H^{[p^r]}$, if their images in $R_H^{\mathfrak{s}_{p^r}}$ are $(\mathbb{Z} / p^r \mathbb{Z})$-linearly independent, then their images in $R_H^{\mathfrak{s}_{p^s}}$ are $(\mathbb{Z} / p^s \mathbb{Z})$-linearly independent.
    \end{enumerate}
\end{lem}

\begin{proof}
    (1) The equivalence of (i) and (ii), and (iii) and (iv) follows immediately from \autoref{lem:CoverAndValue} (1).
    By \autoref{lem:EasyValuationProperties} (2), we have $\mathfrak{s}_p(\overline{m} \cdot \overline{a}) = \mathfrak{s}_{p^r}(p^{r-1} (\overline{m} \cdot \overline{a}))$ for integers $\overline{m}$.
    Thus, the equivalence of (ii) and (iv) follows.

    (2) follows directly from (1), noting that any element in $(H^{[p^r]} + pG) / (H + pG)$ can be represented by an element of $H^{[p^r]}$.
    (3) also follows immediately from (1).
\end{proof}
\subsection{Examples}
\label{sec:ReviewExamples}

In this section, we present several illustrative examples, adapted from \cite[Section 4]{Cluckers2011QuaEli}.
The first example demonstrates that for any linear order $I$ and any $n \geq 2$, there exists an ordered abelian group $G$ such that $\mathcal{S}_n \cong (-\infty) + I$.
Furthermore, for each $H \in \mathcal{S}_n \setminus \left\{ \emptyset \right\}$, we can specify $|R_H^{\mathfrak{s}_n}|$ to be an arbitrary power of $n$ or $\infty$.

\begin{ex}\label{ex:ArbitrarySn}
    Fix $n \geq 2$ and a non-empty linearly ordered set $I$.
    For each positive integer $k$, choose an archimedean ordered abelian group $G_k$ such that $|G_k / nG_k| = n^k$
    (e.g., the multiplicative group $\left\{ p_1^{z_1} \dots p_k^{z_k} \mid z_i \in \mathbb{Z} \right\}$, where $p_i$ is the $i$-th prime).
    Additionally, let $G_{\infty}$ be an archimedean ordered abelian group such that $|G_{\infty} / nG_{\infty}| = \infty$, such as the multiplicative group of positive rational numbers.

    Let $\xi:I \to \mathbb{Z}_{\geq 1} \cup \left\{ \infty \right\}$ be a function, and define $G = \bigoplus_{i \in I} G_{\xi(i)}$ equipped with the reverse lexicographical order.
    That is, for $a = (a_i)_{i \in I} \in G$, we have $a > 0$ if and only if there exists $i \in I$ such that $a_i > 0$ and $a_j = 0$ for all $j > i$.
    Since each $G_k$ is archimedean, every convex subgroup of $G$ is of the form
    \begin{equation*}
        \bigoplus_{i \in I_1} G_{\xi(i)} \oplus \bigoplus_{i \in I_2} \left\{ 0 \right\},
    \end{equation*}
    where $(I_1,I_2)$ is a cut of $I$.

    For $a = (a_i)_{i \in I} \in G \setminus nG$, let $i_0 \in I$ be the index such that $a_{i_0} \notin nG_{\xi(i_0)}$ and $a_i \in nG_{\xi(i)}$ for all $i > i_0$.
    It follows that $\mathfrak{s}_n(a) = \bigoplus_{i<i_0} G_{\xi(i)} \oplus \bigoplus_{i\geq i_0} \left\{ 0 \right\}$.
    Consequently, $\mathcal{S}_n$ is isomorphic to $( -\infty ) + I$ as an ordered set.
    Identifying $i \in I$ with its corresponding element in $\mathcal{S}_n$, we obtain:
    \begin{gather*}
        G_{\leq i}^{\mathfrak{s}_n} = \bigoplus_{j\leq i} G_{\xi(j)} \oplus \bigoplus_{j>i} nG_{\xi(j)}, \text{ and}\\
        G_{< i}^{\mathfrak{s}_n} = \bigoplus_{j<i} G_{\xi(j)} \oplus \bigoplus_{j\geq i} nG_{\xi(j)}.
    \end{gather*}
    Thus, $R_i^{\mathfrak{s}_n} \cong G_{\xi(i)} / nG_{\xi(i)}$, and $|R_i^{\mathfrak{s}_n}| = n^{\xi(i)}$, where $n^{\infty}$ is understood to be $\infty$.
\end{ex}

Recall that every $H \in \mathcal{S}_p \setminus \left\{ \emptyset \right\}$ induces a chain of subgroups
\begin{equation*}
    H + pG \subseteq \dots \subseteq H^{[p^3]} + pG \subseteq H^{[p^2]} + pG \subseteq H^{[p]} + pG = H^{[p]}.
\end{equation*}
The next example demonstrates that the cardinalities $(|(H^{[p^n]} + pG) / (H^{[p^{n+1}]} + pG)|)_{n \geq 1}$ can be prescribed arbitrarily.

\begin{ex}
    Fix $p \in \mathbb{P}$, $m \geq 1$, and an archimedean ordered abelian group $K$.
    We denote by $K^{(\mathbb{Z})}$ the direct sum $\bigoplus_{i \in \mathbb{Z}} K$ equipped with the reverse lexicographical order.
    Let $G = \left\{ (a_i)_{i \in \mathbb{Z}} \in K^{(\mathbb{Z})} \;\middle|\; \sum_{i \in \mathbb{Z}} a_i \in p^m K \right\}$.
    Then $pG = \left\{ (a_i) \in (pK)^{(\mathbb{Z})} \;\middle|\; \sum_{i \in \mathbb{Z}} a_i \in p^{m+1} K \right\}$.
    We show that for every $n \geq 1$, $\left\{ 0 \right\}^{[p^n]} = \left\{ (a_i) \in (p^n K)^{(\mathbb{Z})} \;\middle|\; \sum_{i \in \mathbb{Z}} a_i \in p^m K \right\}$.
    The left-to-right inclusion is immediate.
    Let $(a_i)$ be in the right-hand side.
    Choose an arbitrarily small $i_0 \in \mathbb{Z}$, and consider the convex subgroup $H = \left( \bigoplus_{i \leq i_0} K \oplus \bigoplus_{i > i_0} \left\{ 0 \right\} \right) \cap G$.
    Setting $(b_i) = (- \delta_{i,i_0} \sum_{j \in \mathbb{Z}} a_j)$ ($\delta_{ij}$ is the Kronecker delta), we have $(b_i) \in H$ and $(a_i) + (b_i) \in p^n G$.
    Therefore, $(a_i) \in \left\{ 0 \right\}^{[p^n]}$.
    It follows that
    \begin{equation*}
        \left\{ 0 \right\}^{[p^n]} + pG = \begin{cases}
            \left\{ (a_i) \in (p K)^{(\mathbb{Z})} \;\middle|\; \sum_{i \in \mathbb{Z}} a_i \in p^m K \right\} & (n \leq m)\\
            pG & (m < n)
        \end{cases}.
    \end{equation*}
    Indeed, if $n \leq m$ and $(a_i)$ is in the right-hand side, choose $i_0$ which is less than the support of $(a_i)$ and let $(b_i) = (- \delta_{i,i_0} \sum_{j \in \mathbb{Z}} a_j)$.
    We have $(b_i) \in \left\{ 0 \right\}^{[p^n]}$ and $(a_i) + (b_i) \in p G$.
    As a consequence, we obtain the isomorphism $\left( \left\{ 0 \right\}^{[p^m]} + pG \right) / \left( \left\{ 0 \right\}^{[p^{m+1}]} + pG \right) \cong K / p K$, which is induced by the map $K \to \left\{ 0 \right\}^{[p^m]} + pG$ sending $a \mapsto (\delta_{0,i} p^m a)$.

    More generally, assume that $1 \leq m_1 < m_2$ and that the archimedean ordered abelian group $K$ is decomposed as $K_1 \oplus K_2$ as pure groups (i.e. without considering the ordering).
    For instance, choosing $\mathbb{Q}$-linearly independent $r_1,r_2 \in \mathbb{R}$, we may take $K = r_1 \mathbb{Z} + r_2 \mathbb{Z} \cong \mathbb{Z} \oplus \mathbb{Z}$.
    This time we let $G = \left\{ (a_i,b_i) \in (K_1 \oplus K_2)^{(\mathbb{Z})} \;\middle|\; \sum_{i \in \mathbb{Z}} a_i \in p^{m_1} K_1 \text{ and } \sum_{i \in \mathbb{Z}} b_i \in p^{m_2} K_2 \right\}$.
    By arguing as above, we find that $\left\{ 0 \right\}^{[p^n]} + pG$ equals
    \begin{equation*}
        \left\{ (a_i,b_i) \in (pK_1 \oplus pK_2)^{(\mathbb{Z})} \;\middle|\; \sum_{i \in \mathbb{Z}} a_i \in p^{s_1} K_1 \text{ and } \sum_{i \in \mathbb{Z}} b_i \in p^{s_2} K_2 \right\},
    \end{equation*}
    where
    \begin{equation*}
        (s_1,s_2) = \begin{cases}
            (m_1,m_2) & (1 \leq n \leq m_1)\\
            (m_1+1,m_2) & (m_1 < n \leq m_2)\\
            (m_1+1,m_2+1) & (m_2 < n)
        \end{cases}.
    \end{equation*}
    It then immediately follows $\left( \left\{ 0 \right\}^{[p^{m_1}]} + pG \right) / \left( \left\{ 0 \right\}^{[p^{m_1+1}]} + pG \right) \cong K_1 / p K_1$ and that $\left( \left\{ 0 \right\}^{[p^{m_2}]} + pG \right) / \left( \left\{ 0 \right\}^{[p^{m_2+1}]} + pG \right) \cong K_2 / p K_2$.
    The same argument applies to any finite sequence $1 \leq m_1 < \dots < m_k$.
\end{ex}

See \cite[Chapter 5]{Schmitt1982ModThe} for further examples.
\subsection{Relative quantifier elimination}
\label{sec:ReviewSyn}

Cluckers and Halupczok \cite{Cluckers2011QuaEli} introduced two many-sorted languages, $L_{\mathrm{qe}}$ and $L_{\mathrm{syn}}$.
In addition to the main sort, both languages include auxiliary sorts for $\mathcal{S}_p$, $\mathcal{T}_p$, and $\mathcal{T}_p^+$ for all $p \in \mathbb{P}$.
They proved that quantifiers on the main sort are eliminable in both $L_{\mathrm{qe}}$ and $L_{\mathrm{syn}}$, and that every formula can be transformed into ``family union form''.
This quantifier elimination is relative to the auxiliary sorts, which is optimal because the auxiliary sorts can define any given linear order (as seen in \autoref{ex:ArbitrarySn}).

While Aschenbrenner et al. \cite{Aschenbrenner2022DisVal} provided a distality criterion for ordered abelian groups with finite spines using $L_{\mathrm{qe}}$, we rely on relative quantifier elimination in $L_{\mathrm{syn}}$, which explicitly incorporates the valuations $\mathfrak{s}_p$ and $\mathfrak{t}_p$.
This choice is convenient, since our proof of (non-)distality mainly focuses on the behavior of elements with respect to those valuations.

\begin{dfn}[{\cite[Definition 1.12]{Cluckers2011QuaEli}}]\label{dfn:Lsyn}
    The sorts of the language $L_{\mathrm{syn}}$ consist of the main sort $G$ and the collection of auxiliary sorts $\mathcal{A} = \left\{ \mathcal{S}_p, \mathcal{T}_p, \mathcal{T}_p^+ \mid p \in \mathbb{P} \right\}$.
    The symbols of $L_{\mathrm{syn}}$ are as follows (see \autoref{sec:Notation} for the definition of $\equiv_m$ and $k_H$):
    \begin{enumerate}
        \item $0$, $+$, $-$, $<$, and $\equiv_m$ for each $m \geq 2$ on $G$.
        \item $\leq$ on each $(\alpha,\beta) \in \mathcal{A} \times \mathcal{A}$.
        \item A unary predicate $\mathrm{discr}(x)$ on $\mathcal{S}_p$ for each $p \in \mathbb{P}$, where $\mathrm{discr}(H)$ holds if and only if $G / H$ is discrete.
        \item A unary predicate $\mathrm{zero}(x)$ on each $\alpha \in \mathcal{A}$, where $\mathrm{zero}(H)$ holds if and only if $H = \left\{ 0 \right\}$.\footnote{This predicate was not included in the original definition in \cite{Cluckers2011QuaEli} because $\mathfrak{s}_2(0)$ was defined to be $\left\{ 0 \right\}$ therein. Consequently, the condition $H = \left\{ 0 \right\}$ could be expressed as $(H \geq \mathfrak{s}_2(0)) \land (H \leq \mathfrak{s}_2(0))$.}
        \item Two unary predicates on $\mathcal{S}_p$ for each $p \in \mathbb{P}$, each $s \geq 1$, and each $l \geq 0$ which indicate whether $\mathrm{dim}_{\mathbb{F}_p}\left( (H^{[p^s]} + pG) / (H^{[p^{s+1}]} + pG) \right) = l$ or $\mathrm{dim}_{\mathbb{F}_p}\left( (H^{[p^s]} + pG) / (H + pG) \right) = l$, respectively.
        \item Valuations $\mathfrak{s}_{p^r}:G \to \mathcal{S}_p$ and $\mathfrak{t}_p:G \to \mathcal{T}_p$ for each $p \in \mathbb{P}$ and each $r \geq 1$ (note that $\mathcal{S}_{p^r} = \mathcal{S}_p$).
        \item A unary predicate $x =_{\bullet} k_{\bullet}$ on $G$ for each $k \in \mathbb{Z} \setminus \left\{ 0 \right\}$, meaning there exists $H \Subset G$ such that $G / H$ is discrete and $x / H = k_H$.
        \item A unary predicate $x \equiv_{m,\bullet} k_{\bullet}$ on $G$ for each $m \geq 2$ and each $1 \leq k < m$, meaning there exists $H \Subset G$ such that $G / H$ is discrete and $x / H \equiv_m k_H$.
        \item A unary predicate $D_{p^r}^{[p^s]}(x)$ on $G$ for each $p \in \mathbb{P}$ and each $s \geq r \geq 1$, meaning there exists $H \Subset G$ such that $x \in H^{[p^s]} + p^r G$ and $x \notin H + p^r G$.
    \end{enumerate}
\end{dfn}

The convex subgroup $H$ in (7), (8), and (9) is unique if it exists:

\begin{lem}[{\cite[Lemma 2.12]{Cluckers2011QuaEli}}]\label{lem:UniqueWitness}
    Let $a \in G$, $H \Subset G$, $m \geq 2$, $k \in \mathbb{Z} \setminus \left\{ 0 \right\}$, $p \in \mathbb{P}$, and $s \geq r \geq 1$.
    \begin{enumerate}
        \item If $H$ witnesses $a =_{\bullet} k_{\bullet}$, then $H = \mathfrak{t}_2(a)$.
        \item If $m \nmid k$ and $H$ witnesses $a \equiv_{m,\bullet} k_{\bullet}$, then $H = \mathfrak{s}_m(a)$.
        \item If $H$ witnesses $D_{p^r}^{[p^s]}(a)$, then $H = \mathfrak{s}_{p^r}(a)$.
    \end{enumerate}
\end{lem}

\begin{proof}
    (1) Let $b \in G$ satisfy $b / H = 1_H$.
    Then $b \notin H + 2G$, and for any $H' \Subset G$ larger than $H$, we have $b \in H'$.
    It follows that $H = \mathfrak{s}_2(b) \in \mathcal{S}_2$.
    Since $a \notin H$ and $a \in H'$ for any $H' \in \mathcal{S}_2$ with $H' > H$, $\mathfrak{t}_2(a)$ must equal $H$.

    (2) Since $k_H \notin m(G / H)$, it follows that $a / H \notin m(G / H)$ and thus $a \notin H + mG$.
    On the other hand, $k_H \in H' / H$ for any $H' \Subset G$ with $H' > H$, which implies $a \in H' + mG$.
    Therefore, $\mathfrak{s}_m(a) = H$.

    (3) It is immediate because $H^{[p^s]} + p^r G \subseteq H^{[p^r]}$.
\end{proof}

In particular, these three predicates are definable in $L_{\mathrm{oag}}$.
Consequently, any first-order statement in $L_{\mathrm{syn}}$ can be expressed in $L_{\mathrm{oag}}$.

Note that $x \equiv_m y \leftrightarrow \bigwedge_{p \in \mathbb{P}, p \mid m} x \equiv_{p^{v_p(m)}} y$ by \autoref{lem:CRT} (1).
For the predicate $x \equiv_{m,\bullet} k_{\bullet}$, we have the following decomposition:

\begin{lem}\label{lem:prIsEnough}
    Let $a \in G$, $m \geq 2$, and $1 \leq k < m$.
    We define $P_{\mathrm{n}} = \left\{ p \in \mathbb{P} \;\big|\; p \mid m, v_p(m) > v_p(k) \right\}$ and $P_{\mathrm{d}} = \left\{ p \in \mathbb{P} \;\big|\; p \mid m, v_p(m) \leq v_p(k) \right\}$.
    Then, $a \equiv_{m,\bullet} k_{\bullet}$ if and only if:
    \begin{equation*}
         \bigwedge_{p \in P_{\mathrm{n}}} \left( a \equiv_{p^{v_p(m)},\bullet} k_{\bullet} \right) \land \bigwedge_{p,q \in P_{\mathrm{n}}, r \in P_{\mathrm{d}}} \left( \mathfrak{s}_{p^{v_p}(m)}(a) = \mathfrak{s}_{q^{v_q(m)}}(a) > \mathfrak{s}_{r^{v_r(m)}}(a) \right).
    \end{equation*}
\end{lem}

\begin{proof}
    Suppose $a \equiv_{m,\bullet} k_{\bullet}$ and let $H = \mathfrak{s}_m(a)$.
    This implies $a / H - k_H \in p^{v_p(m)} (G / H)$ for any prime $p \mid m$.
    In particular, $a \equiv_{p^{v_p(m)},\bullet} k_{\bullet}$, and thus $H = \mathfrak{s}_{p^{v_p(m)}}(a)$ for $p \in P_{\mathrm{n}}$.
    For $p \in P_{\mathrm{d}}$, we have $a \in H + p^{v_p(m)} G$, which implies $\mathfrak{s}_{p^{v_p(m)}}(a) < H$.

    We now show the converse.
    Let $H = \mathfrak{s}_{p^{v_p(m)}}(a)$ for some $p \in P_{\mathrm{n}}$ (note that $P_{\mathrm{n}} \neq \emptyset$).
    Then $a / H - k_H \in p^{v_p(m)} (G / H)$ for any prime $p \mid m$, regardless of whether $p$ is in $P_{\mathrm{n}}$ or $P_{\mathrm{d}}$.
    Thus, $a / H - k_H \in m(G / H)$, which means $a \equiv_{m,\bullet} k_{\bullet}$.
\end{proof}

\begin{lem}\label{lem:AddingInfinitesimal}
    Let $a,b \in G$ and $v \in \left\{ \mathfrak{s}_{p^r}, \mathfrak{t}_p \mid p \in \mathbb{P}, r \geq 1 \right\}$.
    If $v(a) > v(b)$, then:
    \begin{itemize}
        \item $v(a + b) = v(a)$.
        \item If $v = \mathfrak{t}_2$ (resp. $\mathfrak{s}_{p^r}$), then $P(a + b) \iff P(a)$, where $P(x)$ is the predicate $x =_{\bullet} k_{\bullet}$ (resp. $x \equiv_{p^r,\bullet} k_{\bullet}$ or $D_{p^r}^{[p^s]}(x)$ for any $s \geq r$).
    \end{itemize}
\end{lem}

\begin{proof}
    The first assertion follows from \autoref{rmk:TriangleIsIsosceles}.
    The second is also immediate from the definitions and \autoref{lem:UniqueWitness}.
\end{proof}

We now present the relative quantifier elimination result by Cluckers and Halupczok.
The following statement, which is sufficient for our arguments on distality, is slightly weaker than the original.

\begin{thm}[{\cite[Theorem 1.13]{Cluckers2011QuaEli}}]\label{thm:RelativeQE}
    Let $\overline{x}$ be $G$-variables and $\overline{\eta}$ be $\mathcal{A}$-variables.
    In the theory of ordered abelian groups, every $L_{\mathrm{syn}}$-formula $\phi(\overline{x},\overline{\eta})$ is equivalent to an $L_{\mathrm{syn}}$-formula of the form:
    \begin{equation*}
        \bigvee_{i=1}^{k} \exists \overline{\theta}\left( \xi_i(\overline{\eta},\overline{\theta}) \land \psi_i(\overline{x},\overline{\theta}) \right),
    \end{equation*}
    where $\overline{\theta}$ are $\mathcal{A}$-variables, each $\xi_i$ involves only $\mathcal{A}$-sort symbols, and each $\psi_i$ is a conjunction of atomic and negated atomic formulas.
\end{thm}

\begin{rmk}\label{rmk:EasierRelativeQE}
    In the setting of \autoref{thm:RelativeQE}, each formula $\exists \overline{\theta}\left( \xi(\overline{\eta},\overline{\theta}) \land \psi(\overline{x},\overline{\theta}) \right)$ can be transformed into an $L_{\mathrm{syn}}$-formula of the form:
    \begin{equation*}
        \xi'(\overline{\eta},\overline{u}(\overline{x})) \land \psi'(\overline{x}),
    \end{equation*}
    where $\xi'(\overline{\eta},\overline{\upsilon})$ involves only $\mathcal{A}$-sort symbols, $\overline{u}(\overline{x})$ is a tuple of terms each of the form $\mathfrak{s}_{p^r}\left( \overline{m} \cdot \overline{x} \right)$ or $\mathfrak{t}_p\left( \overline{m} \cdot \overline{x} \right)$ for a tuple of integers $\overline{m}$, and $\psi'$ is a conjunction of atomic and negated atomic formulas in the $G$-sort.

    Indeed, in $L_{\mathrm{syn}}$, the only symbols connecting the $G$-sort and $\mathcal{A}$-sorts are $\mathfrak{s}_{p^r}$ and $\mathfrak{t}_p$.
    Hence, $\psi(\overline{x},\overline{\theta})$ can be expressed as $\psi_1(\overline{u}(\overline{x}),\overline{\theta}) \land \psi_2(\overline{x})$, where the predicates in $\psi_1$ involve only $\mathcal{A}$-sorts and those in $\psi_2$ involve only the $G$-sort.
    Letting $\xi'(\overline{\eta},\overline{\upsilon})$ be $\exists \overline{\theta}\left( \xi(\overline{\eta},\overline{\theta}) \land \psi_1(\overline{\upsilon},\overline{\theta}) \right)$ and $\psi'(\overline{x})$ be $\psi_2(\overline{x})$ yields the desired form.

    Furthermore, by \autoref{lem:prIsEnough}, we may assume that for each predicate $\equiv_m$ or $\equiv_{m,\bullet} k_{\bullet}$ appearing in $\psi'$, the modulus $m$ is a prime power.
    We can then eliminate $\equiv_{p^r}$ from $\psi'$, as $x \equiv_{p^r} 0$ holds if and only if $\mathfrak{s}_{p^r}(x)$ is the least element of $\mathcal{S}_p$.
\end{rmk}

To prove distality (or non-distality), it is useful to rephrase this theorem as a condition for a map to be elementary.

\begin{prop}\label{prop:FunctionRelativeQE}
    Let $G_1$ and $G_2$ be ordered abelian groups, and let $f = (f_G, (f_{\alpha})_{\alpha \in \mathcal{A}})$ be a family of partial functions, where $f_G:G_1 \to G_2$ and $f_{\alpha}$ is from the sort $\alpha$ in $G_1$ to $\alpha$ in $G_2$ for each $\alpha \in \mathcal{A}$.
    Suppose the following conditions hold:
    \begin{enumerate}
        \item $\mathrm{Dom}(f_G)$ is a subgroup, and its image under $\mathfrak{s}_{p^r}$ (resp. $\mathfrak{t}_p$) is contained in $\mathrm{Dom}(f_{\mathcal{S}_p})$ (resp. $\mathrm{Dom}(f_{\mathcal{T}_p})$).
        \item $f_{\mathcal{S}_p}(\mathfrak{s}_{p^r}(a)) = \mathfrak{s}_{p^r}(f_G(a))$ and $f_{\mathcal{T}_p}(\mathfrak{t}_p(a)) = \mathfrak{t}_p(f_G(a))$ for any $a \in \mathrm{Dom}(f_G)$.
        \item $f_G$ is an ordered group embedding that preserves $x =_{\bullet} k_{\bullet}$, $x \equiv_{p^r,\bullet} k_{\bullet}$, and $D_{p^r}^{[p^s]}(x)$.
        \item The family $(f_{\alpha})_{\alpha \in \mathcal{A}}$ is elementary with respect to the $\mathcal{A}$-sorts.
    \end{enumerate}
    Then $f$ is elementary with respect to the full language $L_{\mathrm{syn}}$.
\end{prop}

\begin{proof}
    For any tuple $\overline{a}$ from $\mathrm{Dom}(f_G)$ and any tuple of integers $\overline{m}$, the assumptions ensure that $f_{\mathcal{S}_p}\left( \mathfrak{s}_{p^r}\left( \overline{m} \cdot \overline{a} \right) \right) = \mathfrak{s}_{p^r}\left( \overline{m} \cdot f_G(\overline{a}) \right)$.
    The same holds for $\mathfrak{t}_p$.
    With this identity established, the claim follows immediately from \autoref{thm:RelativeQE} and \autoref{rmk:EasierRelativeQE}.
\end{proof}

\section{Proof of distality and non-distality}
\label{chp:Distality}

In this section, we prove our main theorem.
\begin{thm}\label{thm:DistalityCriterion}
    An ordered abelian group $G$ is distal if and only if for each $p \in \mathbb{P}$, there exists a positive integer $N_p$ such that for every $H \in \mathcal{S}_p \setminus \left\{ \emptyset \right\}$, $|R_{H}^{\mathfrak{s}_p}| < N_p$.
\end{thm}

The proof is divided into two parts:
In \autoref{sec:Distality}, we show that $G$ is distal if the sizes of its ribs with respect to $\mathfrak{s}_p$ are uniformly bounded for each $p \in \mathbb{P}$.
In \autoref{sec:NonDistality}, we prove that $G$ is not distal if these sizes are not uniformly bounded for some $p \in \mathbb{P}$.

\subsection{Distality}
\label{sec:Distality}

We first establish the distality of the $\mathcal{A}$-sorts.
A \emph{colored order} is a linearly ordered set equipped with a family of unary predicates.
It is well known that every colored order is distal;
specifically, every colored order is dp-minimal (\cite[Proposition 4.2]{Simon2011DpmOrd}), and any dp-minimal structure with a linear order is distal (\cite[Corollary 2.30]{Simon2013DisNon}).
The structure of the $\mathcal{A}$-sorts can be viewed as a union of colored orders.
Although its distality can be inferred from that of colored orders, we provide a direct proof for the reader's convenience.

\begin{lem}\label{lem:AuxiliaryDistality}
    For an ordered abelian group $G$, let $G_{\mathcal{A}}$ be the structure whose sorts consist of all $\mathcal{A}$-sorts with symbols (2), (3), (4), and (5) in \autoref{dfn:Lsyn}.
    Then $G_{\mathcal{A}}$ is distal.
\end{lem}

\begin{proof}
    We may assume $G_{\mathcal{A}}$ is a monster model.
    Let $I,J \models DLO$, $(a_k)_{k \in I+(c)+J}$ be a non-constant indiscernible sequence from a sort $\gamma \in \mathcal{A}$, and $\overline{b} = (b_1,\dots,b_m)$ be such that $(a_k)_{k \in I+J}$ is indiscernible over $\overline{b}$.
    By the indiscernibility of $(a_k)_{k \in I+(c)+J}$, the sequence is strictly monotone, and we may assume it is strictly increasing.
    Furthermore, the indiscernibility of $(a_k)_{k \in I+J}$ over $\overline{b}$ implies that for each $b_i$, either $b_i > (a_k)_{k \in I+J}$ or $b_i < (a_k)_{k \in I+J}$;
    consequently, either $b_i > (a_k)_{k \in I+(c)+J}$ or $b_i < (a_k)_{k \in I+(c)+J}$ holds.

    For any $h < i$ in $I$ and $j$ in $J$, we have $\mathrm{tp}(a_h,a_i,a_j) = \mathrm{tp}(a_h,a_c,a_j)$.
    Since $G_{\mathcal{A}}$ is a monster model, there exists an automorphism $(f_{\alpha})_{\alpha \in \mathcal{A}}$ of $G_\mathcal{A}$ such that $f_{\gamma}((a_h,a_i,a_j)) = (a_h,a_c,a_j)$.
    We define a family of functions $(g_{\alpha})_{\alpha \in \mathcal{A}}$ as follows:
    $g_{\alpha}(x) = f_{\alpha}(x)$ if $a_h \leq x \leq a_j$, and $g_{\alpha}(x) = x$ otherwise.
    Each $g_{\alpha}$ is bijective and preserves all unary predicates.
    Also, $(g_{\alpha})_{\alpha \in \mathcal{A}}$ preserves the orderings.
    Thus, $(g_{\alpha})_{\alpha \in \mathcal{A}}$ is also an automorphism that fixes $\overline{b}$ and maps $a_i$ to $a_c$.
    Therefore, $\mathrm{tp}(a_i / \overline{b}) = \mathrm{tp}(a_c / \overline{b})$, which completes the proof.
\end{proof}

\begin{thm}\label{thm:Distality}
    For an ordered abelian group $G$, suppose that for each $p \in \mathbb{P}$, there exists $N_p > 1$ such that $|R_H^{\mathfrak{s}_p}| < N_p$ for any $H \in \mathcal{S}_p \setminus \left\{ \emptyset \right\}$.
    Then $G$ is distal.
\end{thm}

Before beginning the proof, we establish several preliminary facts.

\begin{lem}\label{lem:RprFinite}
    Suppose that $G$ satisfies the assumption of \autoref{thm:Distality}.
    Then, for each prime $p \in \mathbb{P}$ and each $r \geq 1$, there exists $N_{p^r} > 1$ such that $|R_H^{\mathfrak{s}_{p^r}}| < N_{p^r}$ for any $H \in \mathcal{S}_p \setminus \left\{ \emptyset \right\}$.
\end{lem}

\begin{proof}
    Assume that the conclusion does not hold.
    We may assume $G$ is a monster model.
    By compactness, there exist $p \in \mathbb{P}$, $r \geq 1$, and $H \in \mathcal{S}_p \setminus \left\{ \emptyset \right\}$ such that $|R_H^{\mathfrak{s}_{p^r}}| = \infty$.
    We can then choose a sequence $(a_i)_{i < \omega}$ such that $\mathfrak{s}_{p^r}(a_i) = \mathfrak{s}_{p^r}(a_j - a_i) = H$ for any $i < j <\omega$.
    By Ramsey, we may assume that $(a_i)_{i < \omega}$ is indiscernible.

    Let $s$ be the minimal integer such that $1 \leq s \leq r$ and $\mathfrak{s}_{p^s}(a_j - a_i) = H$ for any $i < j$.
    We define $\widetilde{a_i} = a_i - a_0$ for $i > 0$.
    Then, $\mathfrak{s}_{p^s}(\widetilde{a_i}) = \mathfrak{s}_{p^s}(\widetilde{a_j} - \widetilde{a_i}) = H$ for any $0 < i < j$.
    If $s = 1$, then $|R_H^{\mathfrak{s}_p}| = \infty$, which is a contradiction.
    If $s > 1$, then for any $i > 0$, we have $\mathfrak{s}_{p^{s-1}}(\widetilde{a_i}) < H$ (see \autoref{lem:EasyValuationProperties} (1)).
    Thus, $\widetilde{a_i} = b_i + p^{s-1} a_i'$ for some $b_i \in H$ and $a_i' \in G$.
    As in the proof of \autoref{lem:DifficultValuationProperties} (1), it can be shown that $\mathfrak{s}_p(a_i') = \mathfrak{s}_p(a_j' - a_i') = H$ for any $0 < i < j$.
    This again implies $|R_H^{\mathfrak{s}_p}| = \infty$, a contradiction.
\end{proof}

\begin{rmk}\label{rmk:ValueOfIndiscernibles}
    Let $(a_i)_{i \in I}$ be an indiscernible sequence in the main sort $G$, and let $v \in \left\{ \mathfrak{s}_{p^r}, \mathfrak{t}_p \mid p \in \mathbb{P}, r \geq 1 \right\}$.
    By \autoref{rmk:TriangleIsIsosceles}, exactly one of the following three cases holds:
    \begin{itemize}
        \item $v(a_j - a_i) < v(a_k - a_j) = v(a_i - a_k)$ for any $i < j < k$ in $I$.
        \item $v(a_k - a_j) < v(a_i - a_k) = v(a_j - a_i)$ for any $i < j < k$ in $I$.
        \item $v(a_j - a_i) = v(a_k - a_j) = v(a_i - a_k)$ for any $i < j < k$ in $I$.
    \end{itemize}
    The case $v(a_i - a_k) < v(a_j - a_i) = v(a_k - a_j)$ cannot occur.
    Indeed, If it did, for any $i < j < k < l$ in $I$, we would have $v(a_j - a_i) = v(a_k - a_j)$, $v(a_j - a_i) = v(a_l - a_j)$, and $v(a_j - a_l) < v(a_k - a_j)$, a contradiction.
\end{rmk}

\begin{proof}[Proof of \autoref{thm:Distality}]
    Without loss of generality, we may assume that $G$ is a monster model.
    Fix arbitrary $I,J \models DLO$.
    Let $(\overline{a}_l)_{l \in I+(c)+J}$ be an indiscernible sequence from the main sort $G$ and let $b \in G$, and assume that $(\overline{a}_l)_{l \in I+J}$ is indiscernible over $b$.
    Fix some $i_0 \in I$ and $j_0 \in J$, and let $I' = I_{> i_0}$ and $J' = J_{< j_0}$.
    We will establish the following claim:

    \setcounter{clm}{0}
    \begin{clm}
        For any $m \in \mathbb{Z}$ and any tuple of integers $\overline{m}$ with $|\overline{m}| = |\overline{a}_l|$:
        \begin{itemize}
            \item[(*)] $P(mb - \overline{m} \cdot \overline{a}_c) \iff P(mb - \overline{m} \cdot \overline{a}_l)$ for any $l \in I'+J'$, where $P(x)$ is one of the predicates $x = 0$, $x > 0$, $x =_{\bullet} k_{\bullet}$, $x \equiv_{p^r,\bullet} k_{\bullet}$, or $D_{p^r}^{[p^s]}(x)$.
            \item[(**)] For any $v \in \left\{ \mathfrak{s}_{p^r}, \mathfrak{t}_p \mid p \in \mathbb{P}, r \geq 1 \right\}$, either
            \begin{enumerate}
                \item[(A)] $\left( v\left( mb - \overline{m} \cdot \overline{a}_l \right) \right)_{l \in I'+(c)+J'} = \left( v\left( \overline{m} \cdot \overline{a}_l - \overline{m} \cdot \overline{a}_z \right) \right)_{l \in I'+(c)+J'}$, or
                \item[(B)] $\left( v\left( mb - \overline{m} \cdot \overline{a}_l \right) \right)_{l \in I'+(c)+J'} = \left( v\left( mb - \overline{m} \cdot \overline{a}_z \right) \right)_{l \in I'+(c)+J'}$,
            \end{enumerate}
            where $z \in \left\{ i_0, j_0 \right\}$.
        \end{itemize}
    \end{clm}

    We first see that the goal follows assuming Claim 1.
    Let $V_A$ be the set of triples $(m,\overline{m},v)$---where $m$ and the components of $\overline{m}$ are integers and $v \in \left\{ \mathfrak{s}_{p^r}, \mathfrak{t}_p \mid p \in \mathbb{P}, r \geq 1 \right\}$---that satisfies condition (A).
    We define $V_B$ similarly for condition (B).
    For each $l \in I'+(c)+J'$, we define possibly infinite tuples:
    \begin{gather*}
        \overline{t}_A(\overline{a}_l) = \left( v\left( \overline{m} \cdot \overline{a}_l - \overline{m} \cdot \overline{a}_z \right) \right)_{(m,\overline{m},v) \in V_A};\\
        \overline{t}_B(b) = \left( v\left( mb - \overline{m} \cdot \overline{a}_z \right) \right)_{(m,\overline{m},v) \in V_B},
    \end{gather*}
    where $z = z_{m,\overline{m},v} \in \left\{ i_0,j_0 \right\}$ is the corresponding index.

    In the structure $G_\mathcal{A}$ defined in \autoref{lem:AuxiliaryDistality}, the sequence $\left( \overline{t}_A(\overline{a}_l) \right)_{l \in I'+(c)+J'}$ is indiscernible, while the subsequence $\left( \overline{t}_A(\overline{a}_l) \right)_{l \in I'+J'}$ is indiscernible over $\overline{t}_B(b)$.
    By the distality of $G_\mathcal{A}$, the entire sequence is indiscernible over $\overline{t}_B(b)$.
    Therefore, for any fixed $l \in I'+J'$, we obtain partial functions $(f_{\alpha})_{\alpha \in \mathcal{A}}$ that is elementary in $G_\mathcal{A}$ by setting $f_{\alpha}\left( v\left( mb - \overline{m} \cdot \overline{a}_c \right) \right) = v\left( mb - \overline{m} \cdot \overline{a}_l \right)$ for each triple $(m,\overline{m},v)$.
    We also define $f_G$ by $f_G\left( mb - \overline{m} \cdot \overline{a}_c \right) = mb - \overline{m} \cdot \overline{a}_l$.
    \autoref{prop:FunctionRelativeQE} then implies that $f = (f_G,(f_{\alpha})_{\alpha \in \mathcal{A}})$ is elementary with respect to the full language, and thus $\mathrm{tp}(\overline{a}_c / b) = \mathrm{tp}(\overline{a}_l / b)$, which completes the proof of \autoref{thm:Distality}.

    \begin{proof}[Proof of Claim 1] \renewcommand{\qedsymbol}{$\blacksquare_{\mathrm{Claim}}$}
        We fix integers $m$ and $\overline{m}$.
        First, we establish (*) for the relations $=$ and $>$.
        If the sequence $(\overline{m} \cdot \overline{a}_l)_{l \in I+(c)+J}$ is constant, the result is trivial.
        Otherwise, by the assumptions of indiscernibility, either $mb > (\overline{m} \cdot \overline{a}_l)_{l \in I+(c)+J}$ or $mb < (\overline{m} \cdot \overline{a}_l)_{l \in I+(c)+J}$ holds.
        In either case, (*) holds for $=$ and $>$.

        We now prove (*) for the predicates $=_{\bullet} k_{\bullet}$, $\equiv_{p^r,\bullet} k_{\bullet}$, and $D_{p^r}^{[p^s]}$, as well as (**).
        Let $v$ be either $\mathfrak{s}_{p^r}$ or $\mathfrak{t}_p$.
        Following \autoref{rmk:ValueOfIndiscernibles}, we consider the following cases:

        \setcounter{case}{0}

        \begin{case}
            $\left( v\left( \overline{m} \cdot \overline{a}_l - \overline{m} \cdot \overline{a}_{i_0} \right) \right)_{l \in I'+(c)+J'}$ is strictly increasing.

            By indiscernibility, $v\left( mb - \overline{m} \cdot \overline{a}_{i_0} \right)$ must be either greater than or less than all elements in the sequence $\left( v\left( \overline{m} \cdot \overline{a}_l - \overline{m} \cdot \overline{a}_{i_0} \right) \right)_{l \in I'+(c)+J'}$.

            In the ``greater'' case, \autoref{lem:AddingInfinitesimal} implies $v\left( mb - \overline{m} \cdot \overline{a}_l \right) = v\left( mb - \overline{m} \cdot \overline{a}_{i_0} \right)$ for any $l \in I'+(c)+J'$.
            Also, if $v = \mathfrak{t}_2$ (resp. $\mathfrak{s}_{p^r}$), then $P(mb - \overline{m} \cdot \overline{a}_l) \iff P(mb - \overline{m} \cdot \overline{a}_{i_0})$ for any $l \in I'+(c)+J'$, where $P(x)$ is $x =_{\bullet} k_{\bullet}$ (resp. $x \equiv_{p^r,\bullet} k_{\bullet}$ or $D_{p^r}^{[p^s]}(x)$ for any $s \geq r$).

            In the ``less'' case, we have $v\left( mb - \overline{m} \cdot \overline{a}_l \right) = v\left( \overline{m} \cdot \overline{a}_l - \overline{m} \cdot \overline{a}_{i_0} \right)$ for any $l \in I'+(c)+J'$.
            If $v = \mathfrak{t}_2$ (resp. $\mathfrak{s}_{p^r}$), then $P(mb - \overline{m} \cdot \overline{a}_l) \iff P(\overline{m} \cdot \overline{a}_{i_0} - \overline{m} \cdot \overline{a}_l)$ for any $l \in I'+(c)+J'$, where $P(x)$ is $x =_{\bullet} k_{\bullet}$ (resp. $x \equiv_{p^r,\bullet} k_{\bullet}$ or $D_{p^r}^{[p^s]}(x)$ for any $s \geq r$).
        \end{case}

        \begin{case}
            $\left( v\left( \overline{m} \cdot \overline{a}_l - \overline{m} \cdot \overline{a}_{j_0} \right) \right)_{l \in I'+(c)+J'}$ is strictly decreasing.

            The argument for this case is analogous to Case 1, using $j_0$ instead of $i_0$.
        \end{case}

        \begin{case}
            $v\left( \overline{m} \cdot \overline{a}_{l_1} - \overline{m} \cdot \overline{a}_{i_0} \right) = v\left( \overline{m} \cdot \overline{a}_{l_2} - \overline{m} \cdot \overline{a}_{i_0} \right) = v\left( \overline{m} \cdot \overline{a}_{l_2} - \overline{m} \cdot \overline{a}_{l_1} \right)$ for any $l_1 < l_2$ in $I'+(c)+J'$.

            We first suppose $v = \mathfrak{s}_{p^r}$.
            In this case, $\mathfrak{s}_{p^r}\left( \overline{m} \cdot \overline{a}_l - \overline{m} \cdot \overline{a}_{i_0} \right)$ must be $\emptyset$ for any $l \in I'+(c)+J'$.
            Indeed, if this value were some $H \in \mathcal{S}_p \setminus \left\{ \emptyset \right\}$, then the images of $\left( \overline{m} \cdot \overline{a}_l - \overline{m} \cdot \overline{a}_{i_0} \right)_{l \in I'+(c)+J'}$ in $R_H^{\mathfrak{s}_{p^r}}$ would all be distinct, contradicting $\autoref{lem:RprFinite}$.
            Thus, for any $l \in I'+(c)+J'$ and $s \geq r$, we have $\overline{m} \cdot \overline{a}_l - \overline{m} \cdot \overline{a}_{i_0} \in p^r G$, and hence:
            \begin{itemize}
                \item $\mathfrak{s}_{p^r}\left( mb - \overline{m} \cdot \overline{a}_l \right) = \mathfrak{s}_{p^r}\left( mb - \overline{m} \cdot \overline{a}_{i_0} \right)$.
                \item $P(mb - \overline{m} \cdot \overline{a}_l) \iff P(mb - \overline{m} \cdot \overline{a}_{i_0})$, where $P(x)$ is either $x \equiv_{p^r,\bullet} k_{\bullet}$ or $D_{p^r}^{[p^s]}(x)$.
            \end{itemize}

            Now suppose $v = \mathfrak{t}_p$, and let $H = \mathfrak{t}_p\left( \overline{m} \cdot \overline{a}_l - \overline{m} \cdot \overline{a}_{i_0} \right)$ for an arbitrary $l \in I'+(c)+J'$.
            If $\mathfrak{t}_p\left( mb - \overline{m} \cdot \overline{a}_{i_0} \right) \neq H$, then the argument is the same as in Case 1.
            We therefore assume $\mathfrak{t}_p\left( mb - \overline{m} \cdot \overline{a}_{i_0} \right) = H$.
            Recall that the order on $G$ induces a linear order on the rib $R_H^{\mathfrak{t}_p} = H^+ / H$ (\autoref{lem:CoverAndValue}).
            The images of $\left( \overline{m} \cdot \overline{a}_l - \overline{m} \cdot \overline{a}_{i_0} \right)_{l \in I'+(c)+J'}$ in $R_H^{\mathfrak{t}_p}$ are distinct, and hence strictly increasing or strictly decreasing by indiscernibility.
            Moreover, the image of $mb - \overline{m} \cdot \overline{a}_{i_0}$ is either greater than or less than all these images.
            It follows that $mb - \overline{m} \cdot \overline{a}_l \notin H$ for any $l \in I'+(c)+J'$.
            Equivalently, $\mathfrak{t}_p\left( mb - \overline{m} \cdot \overline{a}_l \right) = H$.
            Similarly, if $v = \mathfrak{t}_2$ and $G / H$ is discrete, then $(mb - \overline{m} \cdot \overline{a}_l) / H \neq k_H$, and hence $mb - \overline{m} \cdot \overline{a}_l \neq_{\bullet} k_{\bullet}$ for any $l \in I'+(c)+J'$.
            \qedhere
        \end{case}
    \end{proof}

    This completes the proof.
\end{proof}
\subsection{Non-distality}
\label{sec:NonDistality}

We first present two preparatory lemmas concerning indiscernible sequences.
Although we present the statements assuming that the index set is $\mathbb{Q}$ for notational simplicity, these lemmas hold for any linearly ordered index set.

\begin{lem}\label{lem:IndisIndependence}
    Let $p \in \mathbb{P}$, $H \in \mathcal{S}_p \setminus \left\{ \emptyset \right\}$, and $r \geq 1$.
    If $(a_i)_{i \in \mathbb{Q}}$ is an indiscernible sequence such that $\mathfrak{s}_{p^r}(a_i) = H$ and $\mathfrak{s}_p(a_j - a_i) = H$ for all $i < j$ in $\mathbb{Q}$, then the equivalent conditions in \autoref{lem:LinearlyIndependentElements} (1) hold for this sequence.
\end{lem}

\begin{proof}
    $a_i \in G_{\leq H}^{\mathfrak{s}_{p^r}} = H^{[p^r]}$ for all $i \in \mathbb{Q}$.
    Let $\overline{m} = (m_1,\dots,m_k)$ be integers such that $p \nmid m_t$ for some index $t$.
    Let $\bar{\imath} = (1,\dots,k)$, and suppose toward a contradiction that $\mathfrak{s}_p(\overline{m} \cdot a_{\bar{\imath}}) < H$.
    Let $\bar{\imath}'$ be the tuple obtained by replacing $t$ in $\bar{\imath}$ with $t + (1 / 2)$.
    By indiscernibility, we have $\mathfrak{s}_p(\overline{m} \cdot a_{\bar{\imath}'}) < H$, which in turn implies $\mathfrak{s}_p(a_{t + (1 / 2)} - a_t) = \mathfrak{s}_p(m_t (a_{t + (1 / 2)} - a_t)) < H$, a contradiction.
\end{proof}

\begin{lem}\label{lem:TotalIndis}
    Let $p \in \mathbb{P}$, $H \in \mathcal{S}_p \setminus \left\{ \emptyset \right\}$, and $r \geq 1$.
    Let $(a_i)_{i \in \mathbb{Q}}$ be an indiscernible sequence from $H^{[p^r]}$ satisfying the equivalent conditions in \autoref{lem:LinearlyIndependentElements} (1).
    Take arbitrary integers $\overline{m} = (m_1,\dots,m_l)$ and a permutation $\sigma$ of $\left\{ 1,\dots,l \right\}$.
    Let $\bar{\imath} = (1,\dots,l)$ and $\sigma(\bar{\imath}) = (\sigma(1),\dots,\sigma(l))$.
    Then:
    \begin{itemize}
        \item $\mathfrak{s}_{p^r}(\overline{m} \cdot a_{\bar{\imath}}) = \mathfrak{s}_{p^r}(\overline{m} \cdot a_{\sigma(\bar{\imath})})$.
        \item $P(\overline{m} \cdot a_{\bar{\imath}}) \iff P(\overline{m} \cdot a_{\sigma(\bar{\imath})})$, where $P(x)$ is either $x \equiv_{p^r,\bullet} k_{\bullet}$ or $D_{p^r}^{[p^s]}(x)$ for any $s \geq r$.\footnote{One can also deduce this from the stability of every abelian group in the pure group language.}
    \end{itemize}
\end{lem}

\begin{proof}
    The first assertion holds because:
    \begin{equation*}
        \mathfrak{s}_{p^r}(\overline{m} \cdot a_{\bar{\imath}}) = \mathfrak{s}_{p^r}(\overline{m} \cdot a_{\sigma(\bar{\imath})}) = 
        \begin{cases}
            H & (\min \left\{ v_p(m_1),\dots,v_p(m_l) \right\} < r)\\
            \emptyset & (\text{otherwise})
        \end{cases}.
    \end{equation*}

    For the second assertion, we may assume $\min \left\{ v_p(m_1),\dots,v_p(m_l) \right\} < r$, as the statement is trivial otherwise.
    Let $t$ be an index such that $v_p(m_t)$ is minimal, and let $\bar{\imath}'$ be the tuple $\bar{\imath}$ with $t$ replaced by $t + (1 / 2)$.

    If $\overline{m} \cdot a_{\bar{\imath}} \equiv_{p^r,\bullet} k_{\bullet}$, then by indiscernibility, $\overline{m} \cdot a_{\bar{\imath}'} \equiv_{p^r,\bullet} k_{\bullet}$.
    This implies $m_t (a_{t + (1/2)} - a_t) \in H + p^r G$, which contradicts $\mathfrak{s}_{p^r}(m_t (a_{t + (1/2)} - a_t)) = H$.
    It follows similarly that $\overline{m} \cdot a_{\sigma(\bar{\imath})} \not\equiv_{p^r,\bullet} k_{\bullet}$.

    Similarly, if $D_{p^r}^{[p^s]}(\overline{m} \cdot a_{\bar{\imath}})$ holds, then $m_t (a_{t + (1/2)} - a_t) \in H^{[p^s]} + p^r G$, and thus $p^{v_p(m_t)} (a_k - a_j) \in H^{[p^s]} + p^r G$ for every $j < k$ in $\mathbb{Q}$.
    Then, $D_{p^r}^{[p^s]}\left( \overline{m} \cdot a_{\sigma(\bar{\imath})} \right) \iff (m_1 + \dots + m_l) a_1 \in H^{[p^s]} + p^r G \iff D_{p^r}^{[p^s]}\left( \overline{m} \cdot a_{\bar{\imath}} \right)$.
    A similar argument shows that $D_{p^r}^{[p^s]}\left( \overline{m} \cdot a_{\sigma(\bar{\imath})} \right)$ implies $D_{p^r}^{[p^s]}(\overline{m} \cdot a_{\bar{\imath}})$.
\end{proof}

\begin{dfn}
    For each $p \in \mathbb{P}$, we say $\mathcal{S}_p$ is \emph{bounded} if there exists $a \in G$ such that $a \notin H$ for all $H \in \mathcal{S}_p$.
\end{dfn}

If $\mathcal{S}_p$ is bounded, then $\mathcal{T}_p$ has a maximum element.
Conversely, if $\mathcal{S}_p$ is unbounded, then $\mathcal{T}_p$ is also unbounded.

\begin{thm}
    Let $G$ be an ordered abelian group.
    Suppose there exists $p \in \mathbb{P}$ such that for every $N \geq 1$, there exists some $H \in \mathcal{S}_p \setminus \left\{ \emptyset \right\}$ satisfying $|R_H^{\mathfrak{s}_p}| > N$.
    Then $G$ is not distal.
\end{thm}

\begin{proof}
    Fix such $p$.
    We may assume that $G$ is a monster model.
    By compactness, there exists $H \in \mathcal{S}_p \setminus \left\{ \emptyset \right\}$ such that $|R_H^{\mathfrak{s}_p}| = \infty$.
    Fix such $H$ and consider the chain of subgroups:
    \begin{equation*}
        H + pG \subseteq \dots \subseteq H^{[p^3]} + pG \subseteq H^{[p^2]} + pG \subseteq H^{[p]} + pG = H^{[p]}.
    \end{equation*}
    Note that $|H^{[p]} / (H + pG)| = |R_H^{\mathfrak{s}_p}| = \infty$.
    We define
    \begin{equation*}
        r = \sup \left\{ s \geq 1 \;\middle|\; |(H^{[p^s]} + pG) / (H + pG)| = \infty \right\},
    \end{equation*}
    where $r$ may be infinite.
    
    \setcounter{clm}{0}
    \begin{clm}
        There exists an indiscernible sequence $(a_i)_{i \in \mathbb{Q}+(\infty)}$ satisfying the following two conditions:
        
        \setcounter{cnd}{0}
        \begin{cnd}
            For any positive integer $s \leq r$, we have $\mathfrak{s}_{p^s}(a_i) = H$ for all $i$, and the images of $(a_i)_{i \in \mathbb{Q}+(\infty)}$ in $R_H^{\mathfrak{s}_{p^s}}$ are $(\mathbb{Z} / p^s \mathbb{Z})$-linearly independent.
        \end{cnd}

        \begin{cnd}
            For any $i < j$ in $\mathbb{Q}+(\infty)$:
            \begin{itemize}
                \item $0 < n a_i < a_j$ for any $n \geq 1$.
                \item For each $q \in \mathbb{P}$:
                \begin{itemize}
                    \item If $\mathcal{S}_q$ is bounded, then $\mathfrak{t}_q(a_i) = \max (\mathcal{T}_q)$.
                    \item If $\mathcal{S}_q$ is unbounded, then $a_i \in \mathfrak{t}_q(a_j)$.
                    \item If $G / \mathfrak{t}_q(a_i)$ is discrete, then $a_i / \mathfrak{t}_q(a_i) > k_{\mathfrak{t}_q(a_i)}$ for any $k \geq 1$.
                \end{itemize}
            \end{itemize}
        \end{cnd}
    \end{clm}

    \begin{proof}[Proof of Claim 1]\renewcommand{\qedsymbol}{$\blacksquare_{\mathrm{Claim}}$}
        It suffices to construct a (possibly non-indiscernible) sequence satisfying these two conditions, as we may then apply Ramsey's theorem to obtain an indiscernible sequence while preserving both conditions.
        By \autoref{lem:LinearlyIndependentElements} and the definition of $r$ (and by compactness if $r = \infty$), it is easy to obtain a sequence $(b_i)_{i \in \mathbb{Q}+(\infty)}$ satisfying Condition 1.
        We use compactness to construct a sequence $(a_i)_{i \in \mathbb{Q}+(\infty)}$ such that $a_i \equiv_{p^s} b_i$ for all $i \in \mathbb{Q}+(\infty)$ and $s \geq 1$, and which satisfies Condition 2.
        If such a sequence is obtained, the congruence condition $a_i \equiv_{p^s} b_i$ ensures that $(a_i)_{i \in \mathbb{Q}+(\infty)}$ also satisfies Condition 1.

        Suppose we have already constructed elements $a_{i_1},\dots,a_{i_l}$.
        Since the coset $b_{i_{l+1}} + p^s G$ is unbounded for any $s \geq 1$, we can choose $a_{i_{l+1}}$ to be sufficiently large  such that $a_{i_{l+1}} \equiv_{p^s} b_{i_{l+1}}$.
        If $G / \mathfrak{t}_q(a_{i_{l+1}})$ is discrete and $a_{i_{l+1}} / \mathfrak{t}_q(a_{i_{l+1}}) = k_{\mathfrak{t}_q(a_{i_{l+1}})}$ for some $q \in \mathbb{P}$ and $k$, we can replace $a_{i_{l+1}}$ with $(p^s + 1) a_{i_{l+1}}$.
        This modification ensures that $a_{i_{l+1}} / \mathfrak{t}_q(a_{i_{l+1}}) > k_{\mathfrak{t}_q(a_{i_{l+1}})}$ while preserving the congruence $a_{i_{l+1}} \equiv_{p^s} b_{i_{l+1}}$.
        By compactness, we obtain the full sequence as required.
    \end{proof}

    Fix a sequence $(a_i)_{i \in \mathbb{Q}+(\infty)}$ given by Claim 1.
    To prove non-distality, we choose an element $a_{\infty}'$ such that $(a_i)_{i \in \mathbb{Q}}$ is not indiscernible over $a_{\infty}'$ and the following claim holds:

    \begin{clm}
        For any tuple of integers $\overline{m} = (m_1,\dots,m_l)$, any nonzero integer $m$, and any strictly increasing indices $\bar{\imath} = (i_1,\dots,i_l)$ from $\mathbb{Q} \setminus \left\{ 0 \right\}$:
        \begin{itemize}
            \item[(*)] $P(\overline{m} \cdot a_{\bar{\imath}} + m a_{\infty}') \iff P(\overline{m} \cdot a_{\bar{\imath}} + m a_{\infty})$, where $P(x)$ is one of the predicates $x = 0$, $x > 0$, $x =_{\bullet} k_{\bullet}$, $x \equiv_{q^s,\bullet} k_{\bullet}$, or $D_{q^s}^{[q^t]}(x)$.
            \item[(**)] $v(\overline{m} \cdot a_{\bar{\imath}} + m a_{\infty}') = v(\overline{m} \cdot a_{\bar{\imath}} + m a_{\infty})$ for any $v \in \left\{ \mathfrak{s}_{q^s}, \mathfrak{t}_q \mid q \in \mathbb{P}, s \geq 1 \right\}$.
        \end{itemize}
    \end{clm}

    We first see that the existence of such $a_{\infty}'$ implies non-distality.
    We define a family of partial functions $f = (f_G, (f_{\alpha})_{\alpha \in \mathcal{A}})$, where $\mathrm{Dom}(f_G)$ is the subgroup generated by $(a_i)_{i \in \mathbb{Q} \setminus \left\{ 0 \right\}}$ and $a_{\infty}'$, with $f_G(\overline{m} \cdot a_{\bar{\imath}} + m a_{\infty}') = \overline{m} \cdot a_{\bar{\imath}} + m a_{\infty}$.
    For each $\alpha \in \mathcal{A}$, we let $f_{\alpha}$ be the identity map on the whole $\alpha$.
    \autoref{prop:FunctionRelativeQE} then implies that $f$ is elementary; hence, $\mathrm{tp}(a_{\infty}' / (a_i)_{i \in \mathbb{Q} \setminus \left\{ 0 \right\}}) = \mathrm{tp}(a_{\infty} / (a_i)_{i \in \mathbb{Q} \setminus \left\{ 0 \right\}})$.
    It implies that $(a_i)_{i \in \mathbb{Q} \setminus \left\{ 0 \right\}}$ is indiscernible over $a_{\infty}'$.
    Since $(a_i)_{i \in \mathbb{Q}}$ itself is not indiscernible over $a_{\infty}'$, it follows that $G$ is not distal.
    We choose $a_{\infty}'$ by considering two cases.

    \setcounter{case}{0}
    \begin{case}
        $r = \infty$.

        We construct $a_{\infty}'$ such that:
        \begin{itemize}
            \item $a_{\infty}' \equiv_{p^s} a_0$ for every $s \geq 1$.
            \item $a_{\infty}' \equiv_{m} a_{\infty}$ for every $m$ coprime with $p$.
            \item $n a_i < a_{\infty}'$ for every $n \geq 1$ and every $i \in \mathbb{Q}$.
            \item $\mathfrak{t}_q(a_{\infty}') = \mathfrak{t}_q(a_{\infty})$ for every $q \in \mathbb{P}$.
            If $G / \mathfrak{t}_q(a_{\infty})$ is discrete, then $a_{\infty}' / \mathfrak{t}_q(a_{\infty}) > k_{\mathfrak{t}_q(a_{\infty})}$ for every $k \geq 1$.
        \end{itemize}
        Let $\mathfrak{t}_{\infty}^+(a_{\infty}) = \bigcap_{q \in \mathbb{P}} \mathfrak{t}_q^+(a_{\infty})$.
        For any $s \geq 1$ and any $m$ coprime with $p$, the intersection $(a_0 + p^s \mathfrak{t}_{\infty}^+(a_{\infty})) \cap (a_{\infty} + m \mathfrak{t}_{\infty}^+(a_{\infty}))$ is a coset of $p^s m \mathfrak{t}_{\infty}^+(a_{\infty})$ by \autoref{lem:CRT} (2).
        We can therefore choose its element that is arbitrarily large in $\mathfrak{t}_{\infty}^+(a_{\infty})$.
        By compactness, we obtain $a_{\infty}'$ with the required properties.
        Since $\mathfrak{s}_p(a_0 - a_{\infty}') = \emptyset$ while $\mathfrak{s}_p(a_i - a_{\infty}') = \mathfrak{s}_p(a_i - a_0) = H$ for $i \in \mathbb{Q} \setminus \left\{ 0 \right\}$, the sequence $(a_i)_{i \in \mathbb{Q}}$ is not indiscernible over $a_{\infty}'$
        
        \begin{proof}[Proof of Claim 2]\renewcommand{\qedsymbol}{$\blacksquare_{\mathrm{Claim}}$}
            Fix $\overline{m}$, $m$, and $\bar{\imath}$.
            We first establish (*) for $x = 0$, $x > 0$, and $x =_{\bullet} k_{\bullet}$, and (**) for $\mathfrak{t}_q$.
            For each $q \in \mathbb{P}$, if $\mathcal{S}_q$ is bounded, let $\widetilde{H} = \max (\mathcal{T}_q)$.
            Condition 2 and the choice of $a_{\infty}'$ imply that $a_{\infty} / \widetilde{H}$ and $a_{\infty}' / \widetilde{H}$ are greater than $n a_i / \widetilde{H}$ and $k_{\widetilde{H}}$ (if $G / \widetilde{H}$ is discrete) for any $n,k \geq 1$ and $i \in \mathbb{Q}$.
            It follows that if $m > 0$, then $(\overline{m} \cdot a_{\bar{\imath}} + m a_{\infty}) / \widetilde{H} > k_{\widetilde{H}}$ for any $k \in \mathbb{Z}$.
            The same holds for $a_{\infty}'$.
            Conversely, if $\mathcal{S}_q$ is unbounded, then $a_i \in \mathfrak{t}_q(a_{\infty}) = \mathfrak{t}_q(a_{\infty}')$ for any $i \in \mathbb{Q}$.
            Thus, $(\overline{m} \cdot a_{\bar{\imath}} + m a_{\infty}) / \mathfrak{t}_q(a_{\infty}) = m a_{\infty} / \mathfrak{t}_q(a_{\infty})$ and $(\overline{m} \cdot a_{\bar{\imath}} + m a_{\infty}') / \mathfrak{t}_q(a_{\infty}) = m a_{\infty}' / \mathfrak{t}_q(a_{\infty})$.
            The required statements follow immediately from these observations.

            Next, we establish (*) for $x \equiv_{q^s,\bullet} k_{\bullet}$ and $D_{q^s}^{[q^t]}(x)$, and (**) for $\mathfrak{s}_{q^s}$.
            If $q$ is not $p$, the statements are trivial since $a_{\infty}' \equiv_{q^s} a_{\infty}$ for all $s \geq 1$.
            Assume $q = p$.
            Since $a_{\infty}' \equiv_{p^s} a_0$ for all $s \geq 1$, it suffices to show the statements with $a_{\infty}'$ replaced by $a_0$, which follow from Condition 1 and \autoref{lem:TotalIndis}.
        \end{proof}
    \end{case}

    \begin{case}
        $r < \infty$.

        \begin{clm}
            For every $s > r$, every tuple of integers $\overline{m} = (m_1,\dots,m_l)$ with $\min \left\{ v_p(m_1),\dots,v_p(m_l) \right\} < s - r$, and every strictly increasing indices $\bar{\imath} = (i_1,\dots,i_l)$ from $\mathbb{Q}+(\infty)$, we have $\mathfrak{s}_{p^s}(\overline{m} \cdot a_{\bar{\imath}}) > H$.
        \end{clm}

        \begin{proof}[Proof of Claim 3] \renewcommand{\qedsymbol}{$\blacksquare_{\mathrm{Claim}}$}
            By \autoref{lem:EasyValuationProperties} (1) and (2), it suffices to show this for $s = r + 1$.
            Suppose $\mathfrak{s}_{p^{r+1}}(\overline{m} \cdot a_{\bar{\imath}}) = H$ toward a contradiction.
            Let $t$ be an index such that $p \nmid m_t$, and let $\bar{\imath}_u$ be $\bar{\imath}$ with $i_t$ replaced by $u \in (i_{t-1},i_t)$.
            Then, $\mathfrak{s}_{p^{r+1}}(\overline{m} \cdot a_{\bar{\imath}_u}) = H$ and $\mathfrak{s}_p(\overline{m} \cdot a_{\bar{\imath}_v} - \overline{m} \cdot a_{\bar{\imath}_u}) = \mathfrak{s}_p(m_t (a_v - a_u)) = H$ for any $u < v$ from $(i_{t-1},i_t)$.
            Since $(\overline{m} \cdot a_{\bar{\imath}_u})_{i_{t-1} < u < i_t}$ is indiscernible, \autoref{lem:IndisIndependence} implies $|(H^{[p^{r+1}]} + pG) / (H + pG)| = \infty$, contradicting the choice of $r$.
        \end{proof}

        By Claim 3 and compactness, there exists $\widetilde{H} \in \mathcal{S}_p$ such that $\widetilde{H} > H$ and $\widetilde{H} \leq \mathfrak{s}_{p^s}(\overline{m} \cdot a_{\bar{\imath}})$ for all $s > r$ and all such $\overline{m}$ and $\bar{\imath}$ as in Claim 3.
        Since $\mathfrak{s}_{p^r}(a_0 - a_{\infty}) = H$ and $H < \bigcap_{q \in \mathbb{P}} \mathfrak{t}_q^+(a_{\infty}) =: \mathfrak{t}_{\infty}^+(a_{\infty})$ (as $a_{\infty}$ is only in the right-hand side), there exists $c \in G$ such that $(a_0 - p^r c) - a_{\infty} \in \widetilde{H} \cap \mathfrak{t}_{\infty}^+(a_{\infty})$.
        As in Case 1 (noting that $a_0 - p^r c \in \mathfrak{t}_{\infty}^+(a_{\infty})$), take $a_{\infty}'$ satisfying:
        \begin{itemize}
            \item $a_{\infty}' \equiv_{p^s} a_0 - p^r c$ for every $s \geq 1$.
            \item $a_{\infty}' \equiv_{m} a_{\infty}$ for every $m$ coprime with $p$.
            \item $n a_i < a_{\infty}'$ for every $n \geq 1$ and every $i \in \mathbb{Q}$.
            \item $\mathfrak{t}_q(a_{\infty}') = \mathfrak{t}_q(a_{\infty})$ for every $q \in \mathbb{P}$. If $G / \mathfrak{t}_q(a_{\infty})$ is discrete, then $a_{\infty}' / \mathfrak{t}_q(a_{\infty}) > k_{\mathfrak{t}_q(a_{\infty})}$ for every $k \geq 1$.
        \end{itemize}
        As in the previous case, $(a_i)_{i \in \mathbb{Q}}$ is not indiscernible over $a_{\infty}'$.

        \begin{proof}[Proof of Claim 2]\renewcommand{\qedsymbol}{$\blacksquare_{\mathrm{Claim}}$}
            Fix $\overline{m}$, $m$, and $\bar{\imath}$.
            Following the same reasoning as in Case 1, it is clear that (*) holds for $x = 0$, $x > 0$, and $x =_{\bullet} k_{\bullet}$,
            and that (**) holds for $\mathfrak{t}_q$.

            We now focus on showing (*) for $x \equiv_{q^s,\bullet} k_{\bullet}$ and $D_{q^s}^{[q^t]}(x)$, as well as (**) for $\mathfrak{s}_{q^s}$.
            Since these conditions are trivial for $q \neq p$, we assume $q = p$.
            If $s \leq r$, we have $a_{\infty}' \equiv_{p^s} a_0$.
            The statements then follow from Condition 1 and \autoref{lem:TotalIndis}.
            Hence, we suppose $s > r$, and let $u = \min \left\{ v_p(m_1),\dots,v_p(m_l),v_p(m) \right\}$.
            If $u \geq s - r$, we use the following identities:
            \begin{itemize}
                \item $\mathfrak{s}_{p^s}(p^{s-r} x) = \mathfrak{s}_{p^r}(x)$.
                \item $\text{If $p^{s-r} \nmid k$, then } p^{s-r} x \not\equiv_{p^s,\bullet} k_{\bullet}$.
                \item $p^{s-r} x \equiv_{p^s,\bullet} (p^{s-r} k)_{\bullet} \iff x \equiv_{p^r,\bullet} k_{\bullet}$.
                \item $D_{p^s}^{[p^t]}(p^{s-r} x) \iff D_{p^r}^{[p^{t-(s-r)}]}(x)$.
            \end{itemize}
            The last one follows from \autoref{rmk:ConsequenceOfConvexityConsequence}.
            By these identities, the problem reduces to the previously established case where $s \leq r$.
            If $u < s - r$, then we have $\mathfrak{s}_{p^s}(\overline{m} \cdot a_{\bar{\imath}} + m a_{\infty}) \geq \widetilde{H}$.
            Since $a_{\infty}' - a_{\infty} = (a_{\infty}' - (a_0 - p^r c)) + ((a_0 - p^r c) - a_{\infty}) \in \widetilde{H} + p^s G$, the assertions follow from \autoref{lem:AddingInfinitesimal}.
        \end{proof}
    \end{case}

    This completes the proof.
\end{proof}

\section{Distal expansions}
\label{chp:DisExp}

\subsection{Refined valuations}

In this chapter, we prove the following theorem:

\begin{thm}\label{thm:DistalExpansion}
    Every ordered abelian group has a distal expansion.
\end{thm}

Our results in \autoref{chp:Distality} imply that the only source of non-distality is the presence of an infinite rib with respect to $\mathfrak{s}_p$.
Thus, our strategy is to refine each $\mathfrak{s}_{p^r}$ so that all the resulting ribs are of size $p$.
We will first describe how such refinements are constructed, and then axiomatize these expansions in \autoref{dfn:RefinedValuationAxioms}.
Fix $p \in \mathbb{P}$, $H \in \mathcal{S}_p \setminus \left\{ \emptyset \right\}$, and a positive integer $N$.
For each $1 \leq k < N$, let $\overline{a}_k \in H^{[p^k]}$ be a (possibly empty, possibly infinite) tuple whose images in $\left( H^{[p^k]} + pG \right) / \left( H^{[p^{k+1}]} + pG \right)$ form an $\mathbb{F}_p$-basis, and let $\overline{a}_N \in H^{[p^N]}$ be a tuple whose images in $\left( H^{[p^N]} + pG \right) / \left( H + pG \right)$ form an $\mathbb{F}_p$-basis.
Recall from \autoref{rmk:ConsequenceOfConvexityConsequence} that $p^s \overline{a}_k \in H^{[p^{k+s}]}$.

\begin{lem}
    In the above setting, the following hold for each $1 \leq r \leq N$:
    \begin{enumerate}
        \item For each $r \leq s < N$, every element in $\left( H^{[p^s]} + p^rG \right) / \left( H^{[p^{s+1}]} + p^rG \right)$ can be uniquely expressed as \begin{equation*}
            \left( \sum_{s-r < k \leq s} p^{s-k} \overline{m}_k \cdot \overline{a}_k \right) / \left( H^{[p^{s+1}]} + p^rG \right),
        \end{equation*}
        where $\overline{m}_k \in \mathbb{F}_p$ (which has only finitely many nonzero components; we omit this remark hereafter).
        \item Every element in $\left( H^{[p^N]} + p^rG \right) / \left( H + p^rG \right)$ can be uniquely expressed as \begin{equation*}
            \left( \sum_{N-r < k \leq N} p^{N-k} \overline{m}_k \cdot \overline{a}_k \right) / \left( H + p^rG \right),
        \end{equation*}
        where $\overline{m}_k \in \mathbb{Z} / p^{k-(N-r)} \mathbb{Z}$.
    \end{enumerate}
\end{lem}

\begin{proof}
    (1) (Existence) Let $a \in H^{[p^s]} + p^rG$.
    By the definition of $\overline{a}_s$, there exist $\overline{m}_s \in \mathbb{F}_p$ such that $a - \overline{m}_s \cdot \overline{a}_s \in H^{[p^{s+1}]} + pG$.
    Choose $a_1 \in G$ such that $a - \overline{m}_s \cdot \overline{a}_s \in H^{[p^{s+1}]} + pa_1$.
    It follows that $pa_1 \in H^{[p^s]} + p^rG$, and hence $a_1 \in H^{[p^{s-1}]} + p^{r-1}G$ by \autoref{rmk:ConsequenceOfConvexityConsequence}.
    Since $a_1 -  \overline{m}_{s-1} \cdot \overline{a}_{s-1} \in H^{[p^s]} + pG $ for some $\overline{m}_{s-1} \in \mathbb{F}_p$, we have $pa_1 - p \overline{m}_{s-1} \cdot \overline{a}_{s-1} \in H^{[p^{s+1}]} + p^2G$, and thus $a - \overline{m}_s \cdot \overline{a}_s - p \overline{m}_{s-1} \cdot \overline{a}_{s-1} \in H^{[p^{s+1}]} + p^2G$.
    Repeating this process yields the desired expression.

    (Uniqueness) Suppose that $\sum_{s-r < k \leq s} p^{s-k} \overline{m}_k \cdot \overline{a}_k \in H^{[p^{s+1}]} + p^rG$.
    Then $\overline{m}_s \cdot \overline{a}_s \in H^{[p^{s+1}]} + pG$, which implies $\overline{m}_s = 0$.
    It follows that $\sum_{s-r < k \leq s-1} p^{s-k} \overline{m}_k \cdot \overline{a}_k \in H^{[p^{s+1}]} + p^rG$, and hence $\sum_{s-r < k \leq s-1} p^{s-1-k} \overline{m}_k \cdot \overline{a}_k \in H^{[p^s]} + p^{r-1}G$.
    The same argument shows that $\overline{m}_{s-1} = 0$, and repeating this reasoning establishes the claim.

    (2) (Existence) Let $a \in H^{[p^N]} + p^rG$. As in (1), we have $a - \overline{m}'_N \cdot \overline{a}_N \in H + pa_1$ for some $0 \leq \overline{m}'_N < p$ and $a_1 \in G$.
    Since $a_1 \in H^{[p^{N-1}]} + p^{r-1}G$, we have $a_1 - \overline{m}'_{N-1} \cdot \overline{a}_{N-1} \in H^{[p^N]} + pG$, and hence $a_1 - \overline{m}'_{N-1} \cdot \overline{a}_{N-1} - \overline{m}''_N \cdot \overline{a}_N \in H + pG$ for some $0 \leq \overline{m}'_{N-1}, \overline{m}''_N < p$.
    Thus, $a - (\overline{m}'_N + p\overline{m}''_N) \cdot \overline{a}_N - p \overline{m}'_{N-1} \cdot \overline{a}_{N-1} \in H + p^2G$.
    Repeat this.

    (Uniqueness) Suppose that $\sum_{N-r < k \leq N} p^{N-k} \overline{m}_k \cdot \overline{a}_k \in H + p^rG$.
    This implies $\overline{m}_N \cdot \overline{a}_N \in H + pG$, which yields $p \mid \overline{m}_N$.
    Then $(\overline{m}_N / p) \cdot \overline{a}_N + \sum_{N-r < k \leq N-1} p^{N-1-k} \overline{m}_k \cdot \overline{a}_k \in H + p^{r-1}G$.
    In particular, $\overline{m}_{N-1} \cdot \overline{a}_{N-1}  \in H^{[p^N]} + pG$.
    Thus $p \mid \overline{m}_{N-1}$, which implies $(\overline{m}_N / p) \cdot \overline{a}_N \in H + pG$, and hence $p \mid (\overline{m}_N / p)$.
    Therefore, $(\overline{m}_N / p^2) \cdot \overline{a}_N + (\overline{m}_{N-1} / p) \cdot \overline{a}_{N-1} + \sum_{N-r < k \leq N-2} p^{N-2-k} \overline{m}_k \cdot \overline{a}_k \in H + p^{r-2}G$.
    Repeat this.
\end{proof}

\begin{cor}\label{cor:LinearRepresentation}
    For each $1 \leq r \leq N$, every element in $R_{H}^{\mathfrak{s}_{p^r}}$ can be uniquely expressed as
    \begin{equation*}
        \left( \sum_{r \leq k \leq N} \sum_{0 \leq s < r} p^s \overline{m}_{k,s} \cdot \overline{a}_k + \sum_{1 \leq k < r} \sum_{r - k \leq s < r} p^s \overline{m}_{k,s} \cdot \overline{a}_k \right) / \left( H + p^rG \right),
    \end{equation*}
    where $0 \leq \overline{m}_{k,s} < p$.
\end{cor}

For each $1 \leq k \leq N$, we arrange $\overline{a}_k$ in a linear order as $\overline{a}_k = (a_{k,i})_{i \in I_k}$ such that each $I_k$ has a maximum element (if nonempty), and if $G / H$ is discrete, then $a_{N,\min I_N} / H = 1_H$.
We introduce a linear order on $\mathbb{N} \times (I_N \cup \dots \cup I_1)$, illustrated in \autoref{fig:NewOrder}, as follows:
Let $(s,i) \in \mathbb{N} \times I_k$ and $(t,j) \in \mathbb{N} \times I_l$.
We define $(s,i) < (t,j)$ if $s + k > t + l$, or if $s + k = t + l$ and $k > l$, or if $s = t$, $k = l$, and $i < j$.
Based on this ordering, we introduce a valuation $\sigma_{p^r}: H^{[p^r]} \to \left( \mathbb{N}_{<r} \times (I_N \cup \dots \cup I_1) \right) \cup \left\{ -\infty \right\}$ for each $1 \leq r \leq N$.

\begin{dfn}\label{dfn:RefinedValuation}
    Let $a \in H^{[p^r]} \setminus ( H + p^rG )$.
    Express $a / ( H + p^rG )$ as in \autoref{cor:LinearRepresentation} and let $\overline{m}_{k,s} = (m_{k,s,i})_{i \in I_k}$.
    We define $\sigma_{p^r}(a)$ to be the maximum $(s,i) \in \mathbb{N}_{<r} \times I_k$ such that $m_{k,s,i} \neq 0$, and write $\sigma_{p^r}(a) = (\sigma_{p^r}^{\mathbb{N}}(a),\sigma_{p^r}^{\Sigma}(a))$.
    If $a \in H + p^rG$, we set $\sigma_{p^r}(a) = -\infty$, which is less than $\mathbb{N}_{<r} \times (I_N \cup \dots \cup I_1)$.
\end{dfn}

\begin{figure}\caption{Order on $\mathbb{N} \times (I_N \cup \dots \cup I_1)$}\label{fig:NewOrder}
    \centering
    \begin{tikzpicture}
        \foreach \x in {1,2,3,4,5,6} \foreach \y in {0,1,2} \node at ($-0.7*(\x,0)-0.6*(0,\y)$) {$I_{\x}$};
        \foreach \x in {1,2,3,4,5,6} \node at ($-0.7*(\x,0)+(0,-1.8)$) {$\vdots$};
        \foreach \y in {0,1,2} \node at ($(-4.9,0)-0.6*(0,\y)$) {$\cdots$};
        \node [anchor=east] at (-6,0) {$r = 0$};
        \node [anchor=east] at (-6,-0.6) {$1$};
        \node [anchor=east] at (-6,-1.2) {$2$};
        \node [anchor=east] at (-6,-1.8) {$\vdots$};
        \begin{scope}[semithick]
            \draw [dotted] (-4.6,-0.8)--(-4.4,-1.0); \draw (-4.4,-1.0)--(-4.0,-1.4); \draw [dotted] (-4.0,-1.4)--(-3.8,-1.6);
            \draw [dotted] (-3.8,-1.6)--(-4.6,-0.2);
            \draw [dotted] (-4.6,-0.2)--(-4.4,-0.4); \draw [->] (-4.4,-0.4)--(-3.85,-0.9); \draw (-3.85,-0.9)--(-3.3,-1.4); \draw [dotted] (-3.3,-1.4)--(-3.1,-1.6);
            \draw [dotted] (-3.1,-1.6)--(-4.4,0.2);
            \draw [->] (-4.4,0.2)--(-3.85,-0.3); \draw (-3.85,-0.3)--(-2.6,-1.4); \draw [dotted] (-2.6,-1.4)--(-2.4,-1.6);
            \draw [dotted] (-2.4,-1.6)--(-3.7,0.2);
            \draw [->] (-3.7,0.2)--(-3.15,-0.3); \draw (-3.15,-0.3)--(-1.9,-1.4); \draw [dotted] (-1.9,-1.4)--(-1.7,-1.6);
            \draw [dotted] (-1.7,-1.6)--(-3.0,0.2);
            \draw [->] (-3.0,0.2)--(-2.45,-0.3); \draw (-2.45,-0.3)--(-1.2,-1.4); \draw [dotted] (-1.2,-1.4)--(-1.0,-1.6);
            \draw [dotted] (-1.0,-1.6)--(-2.3,0.2);
            \draw [->] (-2.3,0.2)--(-1.75,-0.3); \draw (-1.75,-0.3)--(-0.5,-1.4);
            \draw [dotted] (-0.5,-1.4)--(-1.6,0.2);
            \draw [->] (-1.6,0.2)--(-1.05,-0.3); \draw (-1.05,-0.3)--(-0.5,-0.8);
            \draw [dotted] (-0.5,-0.8)--(-0.9,0.2);
            \draw (-0.9,0.2)--(-0.5,-0.2);
        \end{scope}
    \end{tikzpicture}
\end{figure}

We have $\sigma_{p^r}(-a) = \sigma_{p^r}(a)$ and $\sigma_{p^r}(a + b) \leq \max\left\{ \sigma_{p^r}(a),\sigma_{p^r}(b) \right\}$ for all $a,b \in H^{[p^r]}$.
Hence, $\sigma_{p^r}$ is indeed a valuation.
Furthermore, defining ribs as in \autoref{dfn:Ribs}, every rib with respect to $\sigma_{p^r}$ is of size $p$.
We show several properties of these valuations that will be used in the quantifier elimination argument below.

\begin{lem}
    Let $1 \leq r \leq N$.
    We define $G_{\leq (s,i)}^{\sigma_{p^r}} = \left\{ a \in H^{[p^r]} \mid \sigma_{p^r}(a) \leq (s,i) \right\}$ for this lemma, and define $G_{< (s,i)}^{\sigma_{p^r}}$ analogously.
    Then for every $0 \leq s < r$, $1 \leq k \leq N$, $i \in I_k$, and $\square \in \left\{ \leq,< \right\}$, the following hold:
    \begin{itemize}
        \item $(s,i) \in \sigma_{p^r}(H^{[p^r]})$ if and only if $s \geq r - k$.
        \item $G_{\square (r-1,i)}^{\sigma_{p^r}} \cap p^{r-1}G = p^{r-1}G_{\square (0,i)}^{\sigma_p}$.
        \item If $s < r - 1$, choose the maximal $k \leq l < k + (r - 1) - s$ such that $I_l \neq \emptyset$, and let $j = \max I_l$.
        Then $G_{\square (s,i)}^{\sigma_{p^r}} \cap p^{r-1}G = p^{r-1}G_{\leq (0,j)}^{\sigma_p}$.
        \item If $(s,i) \in \sigma_{p^r}(H^{[p^r]})$ and $s < r-1$, then $G_{\square (s,i)}^{\sigma_{p^r}} + p^{r-1}G = G_{\square (s,i)}^{\sigma_{p^{r-1}}}$.
        \item If $I_N \cup \dots \cup I_{k+1} = \emptyset$, then $G_{\square (r-1,i)}^{\sigma_{p^r}} + p^{r-1}G = G_{< H}^{\mathfrak{s}_{p^{r-1}}}$.
        If this set is nonempty, then let $j = \max(I_N + \dots + I_{k+1})$.
        Let $j \in I_l$, and let $t = \max \left\{ (r - 1) - (l - k),0 \right\}$.
        Then $G_{\square (r-1,i)}^{\sigma_{p^r}} + p^{r-1}G = G_{\leq (t,j)}^{\sigma_{p^{r-1}}}$.
    \end{itemize}
\end{lem}

\begin{proof}
    Let $a \in H^{[p^r]} \cap p^{r-1}G$, and write $a = p^{r-1}b$.
    Expressing $b / (H + pG)$ as in \autoref{cor:LinearRepresentation}, we see that $a / (H + p^rG)$ is of the form
    \begin{equation*}
        \left( \sum_{1 \leq k \leq N} p^{r-1} \overline{m}_k \cdot \overline{a}_k \right) / \left( H + p^rG \right).
    \end{equation*}
    This yields the second and third statements.
    For the fourth and fifth statements, let $a \in H^{[p^r]}$ and observe that $a / (H + p^{r-1}G)$ can be written as
    \begin{equation*}
        \left( \sum_{r \leq k \leq N} \sum_{0 \leq s < r-1} p^s \overline{m}_{k,s} \cdot \overline{a}_k + \sum_{2 \leq k < r} \sum_{r - k \leq s < r-1} p^s \overline{m}_{k,s} \cdot \overline{a}_k \right) / \left( H + p^{r-1}G \right). \qedhere
    \end{equation*}
\end{proof}

We now axiomatize these valuations.
For each $z \in \mathbb{Z}$, we define a binary relation $\simeq_z$ on $I_N + \dots + I_1$ by setting $i \simeq_z j$ if $l + z = k$ for $i \in I_k$ and $j \in I_l$.
The structure $(I_N + \dots + I_1;<,(\simeq_z)_{z\in\mathbb{Z}})$ satisfies the following conditions:

\begin{dfn}
    Let $(\Sigma; <)$ be a nonempty linearly ordered set and let $(\simeq_z)_{z \in \mathbb{Z}}$ be a family of binary relations on $\Sigma$.
    We call $(\Sigma; <, (\simeq_z)_{z \in \mathbb{Z}})$ a \emph{good value} if the following conditions holds for all $i, j, k \in \Sigma$ and $z, w \in \mathbb{Z}$:
    \begin{itemize}
        \item $i \simeq_0 i$; $i \simeq_z j \implies j \simeq_{-z} i$; and $i \simeq_z j \land j \simeq_w k \implies i \simeq_{z+w} k$ (in particular, $\simeq_0$ is an equivalence relation).
        \item If $z > 0$, then $i \simeq_z j \implies i < j$.
        \item If $z \geq 0$, then $i \leq k \leq j \land i \simeq_z j \implies \bigvee_{t=0}^z (i \simeq_t k)$ (in particular, each $\simeq_0$-class is convex).
        \item $\Sigma / {\simeq_0}$ is a discrete linear order with endpoints, and each $\simeq_0$-class has a maximum element.
    \end{itemize}
\end{dfn}

Given a good value $\Sigma$, we define a linear order on $\mathbb{N} \times \Sigma$ in a manner analogous to \autoref{fig:NewOrder}:
For $(s,i),(t,j) \in \mathbb{N} \times \Sigma$, if there exists no $z \in \mathbb{Z}$ such that $i \simeq_z j$, we define $(s,i) < (t,j)$ if $i < j$; otherwise (in which case such $z$ is unique), we define $(s,i) < (t,j)$ if $z > t - s$, or if $z = t - s$ and $s < t$, or if $z = t - s$, $s = t$, and $i < j$.
For each $r \geq 1$, the restriction of this order to $\mathbb{N}_{< r} \times \Sigma$ is definable in the language $(<, (\simeq_z)_{z \in \mathbb{Z}})$:
Indeed, if for $(s,i),(t,j) \in \mathbb{N}_{<r} \times \Sigma$ there exists no $-r < z < r$ such that $i \simeq_z j$, then $(s,i) < (t,j)$ if and only if $i < j$.

For each $1 \leq r \leq N$ and $H \in \mathcal{S}_p \setminus \left\{ \emptyset \right\}$, we define a valuation $\sigma_{p^r}^H$ on $H^{[p^r]}$ as in \autoref{dfn:RefinedValuation}.
We then define a valuation $\sigma_{p^r}$ on $G$ as follows:
For $a \in G \setminus p^rG$, let $H = \mathfrak{s}_{p^r}(a)$ and set $\sigma_{p^r}(a) = \sigma_{p^r}^H(a)$, while for $a \in p^rG$, we set $\sigma_{p^r}(a) = -\infty$.
We order these values so that $\mathfrak{s}_{p^r}(a) < \mathfrak{s}_{p^r}(b)$ implies $\sigma_{p^r}(a) < \sigma_{p^r}(b)$.
Furthermore, for each $k \geq 1$, we introduce a unary predicate $M_p^k$ on $\mathcal{S}_p$ such that $M_p^k(H)$ if $k$ is the minimal integer satisfying $H^{[p^{k+1}]} + pG \subsetneq H^{[p^k]} + pG$.
The following properties can be axiomatized by a set of sentences (in the language introduced in \autoref{dfn:LsynPlus} below), and any finite subset of these sentences is satisfied in $G$ by choosing a sufficiently large $N$.

\begin{dfn}\label{dfn:RefinedValuationAxioms}
    Let $(\Sigma_p; <)$ be a nonempty linearly ordered set equipped with binary relations $(\simeq_z)_{z \in \mathbb{Z}}$, and let  $(\sigma_{p^r}:G \to (\mathbb{N}_{<r} \times \Sigma_p) \cup \left\{ -\infty \right\})_{r \geq 1}$ be a family of functions.
    We say that $(\sigma_{p^r})_{r \geq 1}$ are \emph{refined valuations} if the following conditions holds:
    \begin{enumerate}
        \item $\sigma_p^{\Sigma}: G \setminus pG \to \Sigma_p$ is surjective.
        \item $\mathfrak{s}_p(a) < \mathfrak{s}_p(b) \implies \sigma_p^{\Sigma}(a) < \sigma_p^{\Sigma}(b)$ for all $a,b \in G$.
    \end{enumerate}
    For each $H \in \mathcal{S}_p \setminus \left\{ \emptyset \right\}$, let $\Sigma_{p,H} = \sigma_p^{\Sigma}\left( H^{[p]} \setminus (H + pG) \right)$ (note that each $\Sigma_{p,H}$ is a convex subset).
    For every $H, H_1, H_2 \in \mathcal{S}_p \setminus \left\{ \emptyset \right\}$:
    \begin{enumerate}[resume]
        \item If $H_1 \neq H_2$, $i \in \Sigma_{p,H_1}$, and $j \in \Sigma_{p,H_2}$, then $i \not\simeq_z j$ for any $z \in \mathbb{Z}$.
        \item The structure $(\Sigma_{p,H}; <, (\simeq_z)_{z \in \mathbb{Z}})$ is a good value.
    \end{enumerate}
    We equip $\mathbb{N} \times \Sigma_p$ with a linear order by taking the order defined for good values on each $\mathbb{N} \times \Sigma_{p,H}$, and setting $\mathbb{N} \times \Sigma_{p,H_1} < \mathbb{N} \times \Sigma_{p,H_2}$ whenever $H_1 < H_2$.
    For every $r \geq 1$, $H \in \mathcal{S}_p \setminus \left\{ \emptyset \right\}$, and $a\in G$:
    \begin{enumerate}[resume]
        \item $\sigma_{p^r}$ is a valuation such that $\sigma_{p^r}(a) = -\infty \iff a \in p^rG$.
        \item $\sigma_{p^r}^{\Sigma}\left( H^{[{p^r}]} \setminus (H + p^rG) \right) \subseteq \Sigma_{p,H}$.
        \item Every rib with respect to $\sigma_{p^r}$ has size $p$.
        \item $\sigma_{p^{r+1}}(pa) = (\sigma_{p^r}^{\mathbb{N}}(a)+1,\sigma_{p^r}^{\Sigma}(a))$.
        \item If $G / H$ is discrete, then $\sigma_{p^r}(1_H) = (0,\min \Sigma_{p,H})$ (note that $\sigma_{p^r}(1_H)$ is well-defined).
    \end{enumerate}
    Furthermore, for every $0 \leq s < r$, $i \in \Sigma_{p,H}$, and $\square \in \left\{ \leq,< \right\}$:
    \begin{enumerate}[resume]
        \item Let $i_0 = \max \Sigma_{p,H}$. If $M_p^k(H)$ for some $1 \leq k < r$ and $i \simeq_z i_0$ for some $0 \leq z < r - k$, then $(s,i) \in \sigma_{p^r}(H^{[p^r]})$ if and only if $s \geq r - (k + z)$. Otherwise, $(s,i) \in \sigma_{p^r}(H^{[p^r]})$ always holds.
        \item $G_{\square (r-1,i)}^{\sigma_{p^r}} \cap p^{r-1}G = p^{r-1}G_{\square (0,i)}^{\sigma_p}$.
        \item If $s < r - 1$, choose the maximal $0 \leq k < r-1-s$ such that there exists $j$ with $i \simeq_k j$, and choose the maximal $j$.
        Then $G_{\square (s,i)}^{\sigma_{p^r}} \cap p^{r-1}G = p^{r-1}G_{\leq (0,j)}^{\sigma_p}$.
        \item If $(s,i) \in \sigma_{p^r}(H^{[p^r]})$ and $s < r-1$, then $G_{\square (s,i)}^{\sigma_{p^r}} + p^{r-1}G = G_{\square (s,i)}^{\sigma_{p^{r-1}}}$.
        \item If $\left\{ j \in \Sigma_{p,H} \mid j < i\text{ and } j \not\simeq_0 i \right\} = \emptyset$, then $G_{\square (r-1,i)}^{\sigma_{p^r}} + p^{r-1}G = G_{< H}^{\mathfrak{s}_{p^{r-1}}}$.
        If this set is nonempty, then let $j$ be its maximum element.
        If $j \simeq_k i$ for some $1 \leq k \leq r-1$, set $t = r - 1 - k$; otherwise, set $t = 0$.
        Then $G_{\square (r-1,i)}^{\sigma_{p^r}} + p^{r-1}G = G_{\leq (t,j)}^{\sigma_{p^{r-1}}}$.
    \end{enumerate}
\end{dfn}

By compactness, there exists an elementary extension (in $L_{\mathrm{oag}}$) $\widetilde{G}$ of $G$ equipped with refined valuations $(\sigma_{p^r})_{p \in \mathbb{P}, r \geq 1}$.
Therefore, to prove \autoref{thm:DistalExpansion}, it suffices to show that every ordered abelian group with refined valuations is distal.
We define $\pi_{p}:\Sigma_p \to \mathcal{S}_p$ by $\pi_p(\sigma_{p^r}^{\Sigma}(a)) = \mathfrak{s}_{p^r}(a)$ (which is well-defined).
Furthermore, for each $1 \leq k < p^r$, we define $\sigma_{p^r,k}$ as follows:
$\sigma_{p^r,k}(a) = -\infty$ for $a \in p^rG$.
If $a \notin p^rG$, let $H = \mathfrak{s}_{p^r}(a)$.
When $G / H$ is dense, we set $\sigma_{p^r,k}(a) = \sigma_{p^r}(a)$.
When $G / H$ is discrete, we set $\sigma_{p^r,k}(a) = -\infty$ if $a / H \equiv_{p^r} k_H$, and $\sigma_{p^r,k}(a) = \sigma_{p^r}(a - k_H)$ otherwise (which is well-defined and is in $\mathbb{N}_{<r} \times \Sigma_{p,H}$).

\begin{dfn}\label{dfn:LsynPlus}
    We define a language $L_{\mathrm{syn}+}$ by augmenting $L_{\mathrm{syn}}$ with a new sort $\Sigma_{p}$ for each $p \in \mathbb{P}$ and the following symbols:
    \begin{enumerate}
        \item Binary predicates $<$ and $(\simeq_z)_{z\in\mathbb{Z}}$ on $\Sigma_p$ for each $p \in \mathbb{P}$.
        \item A unary predicate $M_p^k$ on $\mathcal{S}_p$ for each $p \in \mathbb{P}$ and $k \geq 1$.
        \item A function $\pi_{p}:\Sigma_p \to \mathcal{S}_p$ for each $p \in \mathbb{P}$.
        \item A function $\sigma_{p^r}:G \to (\mathbb{N}_{<r} \times \Sigma_p) \cup \left\{ -\infty \right\}$ for each $p \in \mathbb{P}$ and $r \geq 1$.\footnote{More rigorously, we introduce functions $\sigma_{p^r}^{\Sigma}:G \to \Sigma_p \cup \left\{ -\infty \right\}$ and unary relations $V_{p^r}^s$ ($0 \leq s < r$), and define $\sigma_{p^r}(a) = (s,i)$ if $a \in V_{p^r}^s$ and $\sigma_{p^r}^{\Sigma}(a) = i$.}
        \item A function $\sigma_{p^r,k}:G \to (\mathbb{N}_{<r} \times \Sigma_p) \cup \left\{ -\infty \right\}$ for each $p \in \mathbb{P}$, $r \geq 1$, and $1 \leq k < p^r$.
    \end{enumerate}
\end{dfn}

We regard an ordered abelian group with refined valuations as an $L_{\mathrm{syn}+}$-structure.
We denote $\mathcal{A} \cup \left\{ \Sigma_p \mid p \in \mathbb{P} \right\}$ by $\mathcal{A}^+$.

\begin{rmk}
    The symbols in \autoref{dfn:Lsyn} (5), (8), (9), and $\mathfrak{s}_{p^r}$ may be omitted from the definition of $L_{\mathrm{syn}+}$:
    $\mathfrak{s}_{p^r}$ is simply given by $\pi_p \circ \sigma_{p^r}^{\Sigma}$.
    For (5) and (9), let $s \geq r \geq 1$ and $H \in \mathcal{S}_p \setminus \left\{ \emptyset \right\}$.
    By \autoref{dfn:RefinedValuationAxioms} (10), (13), and (14), we have $H^{[p^s]} + p^rG = G_{\leq (t,i)}^{\sigma_{p^r}}$ for some $(t,i) \in \mathbb{N}_{<r} \times \Sigma_{p,H}$ or $H^{[p^s]} + p^rG = G_{< H}^{\mathfrak{s}_{p^r}}$.
    Furthermore, if $r = 1$, let $d = |(\Sigma_{p,H})_{\leq i}| \in \mathbb{N} \cup \left\{ \infty \right\}$.
    Then $\mathrm{dim}_{\mathbb{F}_p}\left( (H^{[p^s]} + pG) / (H + pG) \right) = d$ since every rib with respect to $\sigma_p$ is of size $p$.
    $\mathrm{dim}_{\mathbb{F}_p}\left( (H^{[p^s]} + pG) / (H^{[p^{s+1}]} + pG) \right)$ can be computed similarly.
    For (8), one can see that $a \equiv_{p^r,\bullet} k_{\bullet}$ if and only if $G / \mathfrak{s}_{p^r}(a)$ is discrete and $\sigma_{p^r,k}(a) = -\infty$.
\end{rmk}

\begin{rmk}\label{rmk:Sigmaandk}
    For $H \in \mathcal{S}_p \setminus \left\{ \emptyset \right\}$ such that $G / H$ is discrete, let $i_0 = \min \Sigma_{p,H}$.
    It follows from \autoref{dfn:RefinedValuationAxioms} (7), (8), and (9) that for any $k$ such that $p^r \nmid k$, we have $\sigma_{p^r}(k_H) = (v_p(k),i_0)$.
    Therefore, $\sigma_{p^r}(x - k_H) = (v_p(k),i_0)$ if and only if $\mathfrak{s}_{p^r}(x) < H$, or $\mathfrak{s}_{p^r}(x) = H$ and $\sigma_{p^r,k}(x) = (v_p(k),i_0)$.
    For $\gamma \in (\mathbb{N}_{<r} \times \Sigma_{p,H})_{< (v_p(k),i_0)}$, $\sigma_{p^r}(x - k_H) = \gamma$ if and only if $\mathfrak{s}_{p^r}(x) = H$ and $\sigma_{p^r,k}(x) = \gamma$.
    For $\gamma \in (\mathbb{N}_{<r} \times \Sigma_{p,H})_{> (v_p(k),i_0)}$, $\sigma_{p^r}(x - k_H) = \gamma$ if and only if $\sigma_{p^r}(x) = \gamma$.
    Other conditions such as $\sigma_{p^r}(x - k_H) \leq \gamma$ can be expressed in a similar way.
\end{rmk}
\subsection{Quantifier elimination}

We state and prove relative quantifier elimination for our augmented language, analogous to \autoref{prop:FunctionRelativeQE}.
The idea of the proof largely follows that of \cite[Theorem 1.13]{Cluckers2011QuaEli}.
A group-theoretic lemma used there is also required for our argument, so we record it here for convenience.

\begin{lem}[{\cite[Lemma 3.2]{Cluckers2011QuaEli}}]\label{lem:SwissCheese}
    Let $G$ be an abelian group, $G'$ a subgroup of $G$, and $X$ a subset of the form
    \begin{equation*}
        X = (K_0 + d_0) \setminus \bigcup_{i=1}^m (K_i + d_i),
    \end{equation*}
    where each $K_i$ is a subgroup of $G$, $d_i \in G$, and for any distinct $i,j > 0$, we have $K_i + d_i \subseteq K_0 + d_0$ and $(K_i + d_i) \cap (K_j + d_j) = \emptyset$.
    Then, for $x \in G$, we have $x \in X + G'$ if and only if
    \begin{itemize}
        \item $x - d_0 \in K_0 + G'$, and
        \item $\sum_{\left\{ 1\leq i\leq m \mid x - d_i \in K_i + G' \right\}} | (K_0 \cap G') / (K_i \cap G') |^{-1} < 1$,
    \end{itemize}
    where $\infty^{-1}$ is regarded as $0$.
\end{lem}

\begin{prop}\label{prop:RefinedQE}
    Let $G_1$ and $G_2$ be ordered abelian groups with refined valuations, and let $f = (f_G, (f_{\alpha})_{\alpha \in \mathcal{A}^+})$ be a family of partial functions, where $f_G:G_1 \to G_2$ and $f_{\alpha}$ is from the sort $\alpha$ in $G_1$ to $\alpha$ in $G_2$ for each $\alpha \in \mathcal{A}^+$.
    We define $f_{\Sigma_p}(s,i) = (s,f_{\Sigma_p}(i))$ for $(s,i) \in \mathbb{N} \times \Sigma_p$.
    Suppose the following conditions hold:
    \begin{enumerate}
        \item $\mathrm{Dom}(f_G)$ is a subgroup, and its images under $\sigma_{p^r}^{\Sigma}$ and $\sigma_{p^r,k}^{\Sigma}$ (resp. $\mathfrak{t}_p$) are contained in $\mathrm{Dom}(f_{\Sigma_p})$ (resp. $\mathrm{Dom}(f_{\mathcal{T}_p})$).
        \item $\pi_p\left( \mathrm{Dom}(f_{\Sigma_p}) \right) \subseteq \mathrm{Dom}(f_{\mathcal{S}_p})$.
        \item $f_{\Sigma_p}(\sigma_{p^r}(a)) = \sigma_{p^r}(f_G(a))$, $f_{\Sigma_p}(\sigma_{p^r,k}(a)) = \sigma_{p^r,k}(f_G(a))$, and $f_{\mathcal{T}_p}(\mathfrak{t}_p(a)) = \mathfrak{t}_p(f_G(a))$ for any $a \in \mathrm{Dom}(f_G)$.
        \item $f_G$ is an ordered group embedding that preserves $x =_{\bullet} k_{\bullet}$.
        \item The family $(f_{\alpha})_{\alpha \in \mathcal{A}^+}$ is elementary with respect to the $\mathcal{A}^+$-sorts.
    \end{enumerate}
    Then $f$ is elementary with respect to the full language $L_{\mathrm{syn}+}$.
\end{prop}

\begin{proof}
    Observe that by (3), (5), and the definition of $\pi_p$, we have $f_{\mathcal{S}_p}(\mathfrak{s}_{p^r}(a)) = \mathfrak{s}_{p^r}(f_G(a))$ for any $a \in \mathrm{Dom}(f_G)$.
    We show that if $G_2$ is a monster model, then for every element in every sort of $G_1$, there exists an expansion $\widetilde{f} = (\widetilde{f}_G, (\widetilde{f}_{\alpha})_{\alpha \in \mathcal{A}^+})$ of $f$ containing that element in its domain and satisfying the conditions in \autoref{prop:RefinedQE}.
    Once this is proven, a standard back-and-forth argument shows that $f$ is elementary.
    By saturation, we may assume that $\mathrm{Dom}(f_{\alpha})$ is the whole of $\alpha$ for all $\alpha \in \mathcal{A}^+$.
    For $a \in G_1$, let $\mathrm{p}_a^{\mathrm{t}}(x)$ be the set of formulas of the form
    \begin{itemize}
        \item $mx - c \mathrel{\square} 0\ (\square \in \left\{ =,>,< \right\})$,
        \item $\mathfrak{t}_p(mx - c) = H$,
        \item $(mx - c) / H \mathrel{\square} k_H\ (\text{$G / H$ is discrete, } \square \in \left\{ =,\neq \right\})$
    \end{itemize}
    $(p \in \mathbb{P}, m \in \mathbb{Z} \setminus \left\{ 0 \right\}, c \in \mathrm{Dom}(f_G), H \in \mathcal{T}_p, k \in \mathbb{Z} \setminus \left\{ 0 \right\})$ that are satisfied by $a$ (we regard $(mx - c) / H = k_H$ as an abbreviation for $\mathfrak{t}_p(mx - c) = H \land mx - c =_{\bullet} k_{\bullet}$), and let $\mathrm{p}_a^{\mathrm{s}}(x)$ be the set of formulas of the form
    \begin{itemize}
        \item $\sigma_{p^r}(mx - c) = \gamma$ ($\gamma \in (\mathbb{N}_{<r} \times \Sigma_p) \cup \left\{ -\infty \right\}$),
        \item $\sigma_{p^r}(mx - c - k_H) = \gamma$ ($H \in \mathcal{S}_p$, $G / H$ is discrete, $\gamma \in \mathbb{N}_{<r} \times \Sigma_{p,H}$),
        \item $\sigma_{p^r,k}(mx - c) = -\infty$
    \end{itemize}
    ($p \in \mathbb{P}$, $r \geq 1$, $m \in \mathbb{Z} \setminus \left\{ 0 \right\}$, $c \in \mathrm{Dom}(f_G)$, $1 \leq k < p^r$) that are satisfied by $a$ (the second condition can be expressed in $L_{\mathrm{syn}+}$; see \autoref{rmk:Sigmaandk}).
    Let $f(\mathrm{p}_a^{\mathrm{t}})$ and $f(\mathrm{p}_a^{\mathrm{s}})$ be the set of formulas obtained from $\mathrm{p}_a^{\mathrm{t}}$ and $\mathrm{p}_a^{\mathrm{s}}$ by replacing $c$, $H$, and $\gamma$ with $f_G(c)$, $f_{\mathcal{A}}(H)$, and $f_{\Gamma_p}(\gamma)$, respectively.
    It suffices to show that for any $a \in G_1$, there exists $b \in G_2$ satisfying $(f(\mathrm{p}_a^{\mathrm{t}}) \cup f(\mathrm{p}_a^{\mathrm{s}}))(b)$; once this is established, defining the expansion of $f_G$ by $\widetilde{f}_G(ma - c) = mb - f_G(c)$ completes the proof.
    
    \setcounter{clm}{0}
    \begin{clm}
        Let $\mathrm{q}_a^{\mathrm{t}}$ and $\mathrm{q}_a^{\mathrm{s}}$ be the subsets of $\mathrm{p}_a^{\mathrm{t}}$ and $\mathrm{p}_a^{\mathrm{s}}$ consisting of formulas with $m = 1$.
        Suppose that for any $a \in G_1$, there exists $b' \in G_2$ satisfying $(f(\mathrm{q}_a^{\mathrm{t}}) \cup f(\mathrm{q}_a^{\mathrm{s}}))(b')$.
        Then, for any $a \in G_1$, there exists $b \in G_2$ satisfying $(f(\mathrm{p}_a^{\mathrm{t}}) \cup f(\mathrm{p}_a^{\mathrm{s}}))(b)$.
    \end{clm}

    \begin{proof}[Proof of Claim 1] \renewcommand{\qedsymbol}{$\blacksquare_{\mathrm{Claim}}$}
        Consider a formula $\phi \in \mathrm{p}_a^{\mathrm{t}} \cup \mathrm{p}_a^{\mathrm{s}}$ in which the term $mx - c$ appears.
        By \autoref{lem:EasyValuationProperties} and \autoref{dfn:RefinedValuationAxioms} (8), for any $z \neq 0$, there exists a formula $\psi \in \mathrm{p}_a^{\mathrm{t}} \cup \mathrm{p}_a^{\mathrm{s}}$ equivalent to $\phi$ in which $zmx - zc$ appears (note that for $z$ coprime to $p$, we have $\sigma_{p^r}(zx) = \sigma_{p^r}(x)$ by \autoref{dfn:RefinedValuationAxioms} (7)).
        Therefore, for any given finite subset of $\mathrm{p}_a^{\mathrm{t}} \cup \mathrm{p}_a^{\mathrm{s}}$, we may assume that a common $m > 0$ appears in all formulas in that subset.
        We denote this subset by $\Phi(mx)$.
        Then $\left\{ x \equiv_m 0 \right\} \cup \Phi(x) \subseteq \mathrm{q}_{ma}^{\mathrm{t}} \cup \mathrm{q}_{ma}^{\mathrm{s}}$ (note that $x \equiv_m 0 \iff\bigwedge_{p \mid m} (\sigma_{p^{v_p(m)}}(x) = -\infty)$), and hence there exists a solution $b' \in G_2$ to $\left\{ x \equiv_m 0 \right\} \cup f(\Phi)(x)$.
        There exists $b \in G_2$ such that $mb = b'$, which then satisfies $f(\Phi)(mx)$.
    \end{proof}

    \begin{clm}
        For $a \in G_1$, if there exists $b' \in G_2$ satisfying $f(\mathrm{q}_a^{\mathrm{s}})(b')$, then there exists $b \in G_2$ satisfying $(f(\mathrm{q}_a^{\mathrm{t}}) \cup f(\mathrm{q}_a^{\mathrm{s}}))(b)$.
    \end{clm}

    \begin{proof}[Proof of Claim 2] \renewcommand{\qedsymbol}{$\blacksquare_{\mathrm{Claim}}$}
        We may assume that $a \notin \mathrm{Dom}(f_G)$.
        Choose finite subsets $\Phi^{\mathrm{t}} \subseteq \mathrm{q}_a^{\mathrm{t}}$ and $\Phi^{\mathrm{s}} \subseteq \mathrm{q}_a^{\mathrm{s}}$ arbitrarily, and let $n$ be such that $p^r \mid n$ for all $p^r$ appearing in $\Phi^{\mathrm{t}} \cup \Phi^{\mathrm{s}}$.
        Note that $(x \equiv_n y) \rightarrow (\Phi^{\mathrm{s}}(x) \leftrightarrow \Phi^{\mathrm{s}}(y))$, and that the same holds for $f(\Phi^{\mathrm{s}})$.
        It suffices to find $\Psi \subseteq \mathrm{q}_a^{\mathrm{s}}$ such that $f(\Psi)(x) \rightarrow \exists y (x \equiv_n y \land f(\Phi^{\mathrm{t}})(y))$, since we can then choose $b' \in (f(\Psi) \cup f(\Phi^{\mathrm{s}}))(G_2)$.

        First, let $c \in \mathrm{Dom}(f_G)$ and $p \mid n$.
        Set $K = \mathfrak{t}_p(a - c)$, $K^+ = \mathfrak{t}_p^+(a - c)$, $H = \mathfrak{t}_n(a - c)$, and $H^+ = \mathfrak{t}_n^+(a - c)$.
        By \autoref{lem:DifficultValuationProperties} (4), we have $K \leq H < H^+ \leq K^+$, which implies that $\mathfrak{t}_n(x - c) = H \rightarrow \mathfrak{t}_p(x - c) = K$.
        Furthermore, consider the case where $G_1 / K$ is discrete.
        If $K < H$, then $\mathfrak{t}_n(x - c) = H \rightarrow (x - c) / K \neq k_{K}$ for every $k \neq 0$.
        If $K = H$, then $(x - c) / H = k_H \leftrightarrow (x - c) / K = k_{K}$.
        These formulas also hold when $c$, $H$, and $K$ are replaced by $f_G(c)$, $f_{\mathcal{A}}(H)$, and $f_{\mathcal{A}}(K)$.

        By this observation, there exists a finite set $\Phi'$ consisting of formulas
        \begin{itemize}
            \item $x - c \mathrel{\square} 0\ (\square \in \left\{ >,< \right\})$,
            \item $\mathfrak{t}_n(x - c) = H$,
            \item $(x - c) / H \mathrel{\square} k_H\ (\text{$G / H$ is discrete, } \square \in \left\{ =,\neq \right\})$
        \end{itemize}
        $(c \in \mathrm{Dom}(f_G), H = \mathfrak{t}_n(a - c), k \in \mathbb{Z} \setminus \left\{ 0 \right\})$ such that $a \in \Phi'(G_1) \subseteq \Phi^{\mathrm{t}}(G_1)$ and $f(\Phi')(G_2) \subseteq f(\Phi^{\mathrm{t}})(G_2)$.
        Let $c_1,\dots,c_m$ be the elements appearing in $\Phi'$, and set $H_i = \mathfrak{t}_n(a - c_i)$.
        If $H_i < H_j$, then $\mathfrak{t}_n(c_i - c_j) = H_j$, and thus $\mathfrak{t}_n(x - c_i) = H_i \rightarrow \mathfrak{t}_n(x - c_j) = H_j \land (x - c_j) / H_j = (c_i - c_j) / H_j$.
        The same holds when the elements are mapped by $f$.
        Hence, we may assume that $H_1 = \dots = H_m$, and denote this by $H$.\footnote{If $H = \emptyset$, we replace, say, $c_l / H$ and $a / H$ with $c_l$ and $a$. Since the predicate zero in \autoref{dfn:Lsyn} (4) is included in the language, $H^+ > \left\{ 0 \right\}$ implies $f(H)^+ > \left\{ 0 \right\}$,  so the subsequent argument remains valid.}
        Choose (if they exist) the maximal $c_l$ and the minimal $c_r$ such that $c_l / H < a / H < c_r / H$.
        We assume that $c_l$ exists (the case where only $c_r$ exists is handled symmetrically).
        \begin{itemize}
            \item If $G / H$ is dense, let $\Phi''$ be $\mathfrak{t}_n(x - c_l) = H \land c_l / H < x / H < c_r / H$,  and let $\widetilde{\Psi}$ be $\mathfrak{s}_n(x - c_l) \leq H$.
            \item If $G / H$ is discrete and $c_l / H + k_H < a / H < c_r / H - k_H$ for all $k > 0$, let $\Phi''$ be $\mathfrak{t}_n(x - c_l) = H \land c_l / H + k_H < x / H < c_r / H - k_H$ for a sufficiently large $k$, and let $\widetilde{\Psi}$ be $\mathfrak{s}_n(x - c_l) \leq H$.
            \item If $G / H$ is discrete and $a / H = c_l / H + k_H$ for some $k > 0$ (the case for $c_r$ is analogous), let $\Phi''$ be $(x - c_l) / H = k_H$.
            If $n \mid k$, let $\widetilde{\Psi}$ be $\mathfrak{s}_n(x - c_l) < H$; if $n \nmid k$, let $\widetilde{\Psi}$ be $(x - c_l) / H \equiv_n k_H$.
        \end{itemize}
        In all cases, we have $\Phi''(G_1) \subseteq \Phi'(G_1)$ and $\widetilde{\Psi}(x) \rightarrow \exists y (x \equiv_n y \land \Phi''(y))$.
        Indeed, in the first two cases, if $\mathfrak{s}_n(x - c_l) \leq H$, then there exists $z$ such that $x + nz - c_l \in H^+$.
        By the $n$-regularity of $H^+ / H$, there exists $w$ satisfying $\Phi''(x + nw)$.
        The last case is verified similarly.
        The same also holds for $f(\Phi'')$, $f(\Phi')$, and $f(\widetilde{\Psi})$.
        Finally, there exists a finite subset $\Psi \subseteq \mathrm{q}_a^{\mathrm{s}}$ such that $\Psi(G_1) \subseteq \widetilde{\Psi}(G_1)$ and $f(\Psi)(G_2) \subseteq f(\widetilde{\Psi})(G_2)$; for instance, $\mathfrak{s}_n(x - c_l) \leq H$ is implied by $\bigwedge_{p \mid n} (\sigma_{p^{v_p(n)}}(x - c_l) = \gamma_p)$, where $\gamma_p = \sigma_{p^{v_p(n)}}(a - c_l)$.
    \end{proof}

    (*) Let $\phi(x)$ be one of the following formulas:
    \begin{itemize}
        \item $x = x$,
        \item $\sigma_{p^r}(x - c) \mathrel{\square} \gamma$ ($\gamma \in \mathbb{N}_{<r} \times \Sigma_p$),
        \item $\sigma_{p^r}(x - c - k_H) \mathrel{\square} \gamma$ ($H \in \mathcal{S}_p$, $G / H$ is discrete, $\gamma \in \mathbb{N}_{<r} \times \Sigma_{p,H}$),
        \item $(x - c) / H \equiv_{p^r} k_H$ ($H \in \mathcal{S}_p$, $G / H$ is discrete)
    \end{itemize} 
    ($p \in \mathbb{P}$, $r \geq 1$, $c \in \mathrm{Dom}(f_G)$, $1 \leq k < p^r$, $\square \in \left\{ \leq, < \right\}$).
    $\phi(G_1)$ can be expreesed as a coset $K + d$, where $K$ is one of the subgroups $G_1$, $(G_1)_{\leq \gamma}^{\sigma_{p^r}}$, $(G_1)_{< \gamma}^{\sigma_{p^r}}$, or $(G_1)_{< H}^{\mathfrak{s}_{p^r}}$, and $d$ is either $c$ or $c + k_H$.
    We denote $f(\phi)(G_2)$ by $K' + d'$.

    \begin{clm}
        For any two cosets $K_1 + d_1$ and $K_2 + d_2$ in (*) with common $p$ and $r$, we have $K_1 + d_1 \subseteq K_2 + d_2$ if and only if $K'_1 + d'_1 \subseteq K'_2 + d'_2$ (note that any two such cosets are either disjoint or one is contained in the other).
    \end{clm}

    \begin{proof}[Proof of Claim 3]
        We may assume that $K_1 \subseteq K_2$ (or equivalently, $K_1' \subseteq K_2'$).
        Then $K_1 + d_1 \subseteq K_2 + d_2$ if and only if $d_1 - d_2 \in K_2$.
        Writing $c_i$ as $c_i + 0_{H_i}$, if necessary, each $d_i$ is of the form $c_i + k_{i,H_i}$.
        For simplicity, suppose that $K_2= (G_1)_{\leq \gamma_2}^{\sigma_{p^r}}$ (the other cases are analogous).
        Then, $d_1 - d_2 \in K_2$ if and only if $\sigma_{p^r}((c_1 - c_2) + (k_1 - k_2)_{H_2}) \leq \gamma_2$ 
        (if $H_1 < H_2$, we may replace $c_1 + k_{1,H_1}$ with $c_1 + 0_{H_2}$).
        By \autoref{rmk:Sigmaandk}, $f$ preserves this condition.
    \end{proof}
    
    \begin{clm}
        For any $a \in G_1$, there exists $b \in G_2$ satisfying $f(\mathrm{q}_a^{\mathrm{s}})(b)$.
    \end{clm}

    \begin{proof}[Proof of Claim 4] \renewcommand{\qedsymbol}{$\blacksquare_{\mathrm{Claim}}$}
        Let $\Phi(x)$ be an arbitrary finite subset of $\mathrm{q}_a^{\mathrm{s}}$, and decompose it as $\Phi_{p_1} \cup \dots \cup \Phi_{p_m}$ according to the prime numbers appearing in the formulas.
        By \autoref{lem:CRT} (2), it suffices to find a solution to each $f(\Phi_{p_i})$ separately.
        Thus, we may assume that a single prime $p$ appears in $\Phi$.
        Choose $r \geq 1$ such that $\Phi = \Phi_{<r} \cup \Phi_r$, where the exponents of $p$ appearing in $\Phi_{<r}$ are less than $r$, and those in $\Phi_r$ are $r$.

        $\Phi_r(G_1)$ is equivalent to a boolean combination of cosets in (*).
        Thus, we can find a finite set of formulas $\Psi$ defining $(K_0 + d_0) \setminus \bigcup_{i=1}^m (K_i + d_i)$, where $K_i + d_i \subseteq K_0 + d_0$ and $(K_i + d_i) \cap (K_j + d_j) = \emptyset$ for any distinct $i,j > 0$, such that $a \in \Psi(G_1) \subseteq \Phi_r(G_1)$ and $f(\Psi)(G_2) \subseteq f(\Phi_r)(G_2)$.
        \begin{itemize}
            \item If $r = 1$, then $|K_0 / K_i| = |K'_0 / K'_i|$ for every $i \leq m$.
            For instance, if $K_0 = (G_1)_{\leq (0,j_0)}^{\sigma_p}$ and $K_i = (G_1)_{\leq (0,j_i)}^{\sigma_p}$ for $j_0,j_i \in \Sigma_p$, let $d = |(j_i,j_0]| \in \mathbb{N} \cup \left\{ \infty \right\}$.
            It follows that $d = |(f_{\Sigma_p}(j_i),f_{\Sigma_p}(j_0)]|$, and that $|K_0 / K_i| = p^d = |K'_0 / K'_i|$.
            \item Since $G_{<H}^{\mathfrak{s}_{p^r}} \cap p^{r-1}G = p^{r-1}G_{<H}^{\mathfrak{s}_p}$ and by \autoref{dfn:RefinedValuationAxioms} (11) and (12), for every $i$ there exists a subgroup $L_i$ in (*) with $r = 1$ such that $K_i \cap p^{r-1}G_1 = p^{r-1}L_i$ and $K'_i \cap p^{r-1}G_2 = p^{r-1}L'_i$.
            Because the torsion-freeness of $G_1$ implies $p^{r-1}L_0 / p^{r-1}L_i \cong L_0 / L_i$, we have $|(K_0 \cap p^{r-1}G_1) / (K_i \cap p^{r-1}G_1)| = |(K'_0 \cap p^{r-1}G_2) / (K'_i \cap p^{r-1}G_2)|$.
            \item Since $G_{<H}^{\mathfrak{s}_{p^r}} + p^{r-1}G = G_{<H}^{\mathfrak{s}_{p^{r-1}}}$ and by \autoref{dfn:RefinedValuationAxioms} (13) and (14), for every $i$ there exists a coset $L_i + e_i$ in (*) with exponent $r-1$ such that $K_i + p^{r-1}G_1 + d_i = L_i + e_i$ and $K'_i + p^{r-1}G_2 + d'_i = L'_i + e'_i$.
        \end{itemize}
        We now proceed by induction on $r$.
        If $r = 1$, \autoref{lem:SwissCheese} implies that $\Psi(G_1) \neq \emptyset \iff \sum_{1\leq i\leq m} | K_0 / K_i |^{-1} < 1 \iff \sum_{1\leq i\leq m} | K'_0 / K'_i |^{-1} < 1 \iff f(\Psi)(G_2) \neq \emptyset$, and hence $f(\Phi_1)(G_2) \neq \emptyset$.
        Suppose $r > 1$.
        It suffices to find $b' \in f(\Phi_{<r})(G_2) \cap (f(\Phi_r)(G_2) + p^{r-1}G_2)$.
        Let $I_{\top}$ be the set of indices $1 \leq i \leq m$ such that $a \in K_i + p^{r-1}G_1 + d_i$, and let $I_{\bot} = \left\{ 1,\dots,m \right\} \setminus I_{\top}$.
        By \autoref{lem:SwissCheese}, $\sum_{i \in I_{\top}} |(K_0 \cap p^{r-1}G_1) / (K_i \cap p^{r-1}G_1)|^{-1} < 1$.
        Let $\Psi'(x)$ be a set of formulas defining $(K_0 + p^{r-1}G_1 + d_0) \setminus \bigcup_{i \in I_{\bot}} (K_i + p^{r-1}G_1 + d_i)$.
        Again by \autoref{lem:SwissCheese}, $a \in \Psi'(G_1) \subseteq \Psi(G_1) + p^{r-1}G_1$ and $f(\Psi')(G_2) \subseteq f(\Psi)(G_2) + p^{r-1}G_2$.
        Since $\Psi'$ is implied by finitely many formulas in $\mathrm{q}_a^{\mathrm{s}}$ in which only exponent $r-1$ appears, the induction hypothesis applies.
    \end{proof}

    It completes the proof.
\end{proof}
\subsection{Proof of distality}

\begin{lem}\label{lem:RefinedAuxiliaryDistality}
    For an ordered abelian group $G$ with refined valuations, let $G_{\mathcal{A}^+}$ be the structure whose sorts consist of all $\mathcal{A}^+$-sorts with symbols (2), (3), and (4) in \autoref{dfn:Lsyn} and (1), (2), and (3) in \autoref{dfn:LsynPlus}.
    Then $G_{\mathcal{A}^+}$ is distal.
\end{lem}

\begin{proof}
    We modify the proof of \autoref{lem:AuxiliaryDistality}.
    We may assume $G_{\mathcal{A}^+}$ is a monster model.
    Let $I,J \models DLO$, $(a_k)_{k \in I+(c)+J}$ be a non-constant indiscernible sequence from a sort $\gamma \in \mathcal{A}^+$, and $\overline{b} = (b_1,\dots,b_m)$ be such that $(a_k)_{k \in I+J}$ is indiscernible over $\overline{b}$.
    We may assume the sequence is strictly increasing.
    For any $h < i$ in $I$ and $j$ in $J$, there exists an automorphism $(f_{\alpha})_{\alpha \in \mathcal{A}^+}$ of $G_{\mathcal{A}^+}$ such that $f_{\gamma}((a_h,a_i,a_j)) = (a_h,a_c,a_j)$.
    We define a family of functions $(g_{\alpha})_{\alpha \in \mathcal{A}^+}$ as follows:
    \begin{itemize}
        \item If $a_k$ is in an $\mathcal{A}$-sort, then $g_{\alpha}(x) = f_{\alpha}(x)$ if ($x$ is in an $\mathcal{A}$-sort and $a_h < x < a_j$) or ($x \in \Sigma_q$ for some $q$ and $a_h < \pi_q(x) < a_j$), and $g_{\alpha}(x) = x$ otherwise.
        \item If $a_k \in \Sigma_p$ for some $p$, then $g_{\alpha}(x) = f_{\alpha}(x)$ if ($x \in \Sigma_p$ and $a_h < x < a_j$), ($x \in \Sigma_q$ for some $q \neq p$ and $\pi_p(a_h) < \pi_q(x) < \pi_p(a_j)$), or ($x$ is in an $\mathcal{A}$-sort and $\pi_p(a_h) < x < \pi_p(a_j)$), and $g_{\alpha}(x) = x$ otherwise.
    \end{itemize}
    It is straightforward to see that $(g_{\alpha})_{\alpha \in \mathcal{A}^+}$ are all bijective, fixes $\overline{b}$, maps $a_i$ to $a_c$, and preserves the orderings as well as all unary predicates.
    It remains to verify that $(g_{\alpha})_{\alpha \in \mathcal{A}^+}$ preserve the relations $(\simeq_z)_{z\in\mathbb{Z}}$.
    Let $x,y \in \Sigma_q$.
    If $\pi_q(x) \neq \pi_q(y)$, then $x \not\simeq_z y$ for all $z$ by \autoref{dfn:RefinedValuationAxioms} (3).
    Since $\pi_q(x) = \pi_q(y) \iff \pi_q(g_{\Sigma_q}(x)) = \pi_q(g_{\Sigma_q}(y))$, we may assume that $\pi_q(x) = \pi_q(y)$.
    The only nontrivial case is when $a_k \in \Sigma_p$ and $p = q$, since otherwise $(g_{\Sigma_q}(x), g_{\Sigma_q}(y)) = (x, y)$ or $(g_{\Sigma_q}(x), g_{\Sigma_q}(y)) = (f_{\Sigma_q}(x), f_{\Sigma_q}(y))$.
    Suppose $a_h < x < a_j \leq y$ (the other cases are treated similarly).
    Note that $x \simeq_z a_j \iff g_{\Sigma_p}(x) \simeq_z a_j$, and likewise for $y$.
    If there exists $k \geq 0$ such that $x \simeq_k a_j$, then $x \simeq_z y \iff a_j \simeq_{z-k} y \iff a_j \simeq_{z-k} g_{\Sigma_p}(y) \iff g_{\Sigma_p}(x) \simeq_z g_{\Sigma_p}(y)$.
    If no such $k$ exists, then necessarily $x \not\simeq_z y$ and $g_{\Sigma_p}(x) \not\simeq_z g_{\Sigma_p}(y)$ for all $z \in \mathbb{Z}$.
\end{proof}

\begin{thm}
    An ordered abelian group $G$ with refined valuations is distal.
\end{thm}

\begin{proof}
    We may assume that $G$ is a monster model.
    Fix arbitrary $I,J \models DLO$.
    Let $(\overline{a}_l)_{l \in I+(c)+J}$ be an indiscernible sequence from the main sort $G$, let $b \in G$, and assume that $(\overline{a}_l)_{l \in I+J}$ is indiscernible over $b$.
    Fix some $i_0 \in I$ and $j_0 \in J$, and let $I' = I_{> i_0}$ and $J' = J_{< j_0}$.
    We will establish the following claim:

    \setcounter{clm}{0}
    \begin{clm}
        For any $m \in \mathbb{Z}$ and any tuple of integers $\overline{m}$ with $|\overline{m}| = |\overline{a}_l|$:
        \begin{itemize}
            \item[(*)] $P(mb - \overline{m} \cdot \overline{a}_c) \iff P(mb - \overline{m} \cdot \overline{a}_l)$ for any $l \in I'+J'$, where $P(x)$ is one of the predicates $x = 0$, $x > 0$, or $x =_{\bullet} k_{\bullet}$.
            \item[(**)] For any $v \in \left\{ \sigma_{p^r}, \mathfrak{t}_p \mid p \in \mathbb{P}, r \geq 1 \right\}$, either
            \begin{itemize}
                \item[(A)] $\left( v\left( mb - \overline{m} \cdot \overline{a}_l \right) \right)_{l \in I'+(c)+J'} = \left( v\left( \overline{m} \cdot \overline{a}_z - \overline{m} \cdot \overline{a}_l \right) \right)_{l \in I'+(c)+J'}$, or
                \item[(B)] $\left( v\left( mb - \overline{m} \cdot \overline{a}_l \right) \right)_{l \in I'+(c)+J'} = \left( v\left( mb - \overline{m} \cdot \overline{a}_z \right) \right)_{l \in I'+(c)+J'}$,
            \end{itemize}
            where $z \in \left\{ i_0, j_0 \right\}$.
            (The same for $\mathfrak{s}_{p^r}$ automatically follows.)
            \item [(***)] For any $p \in \mathbb{P}$, $r \geq 1$, and $1 \leq k < p^r$, $\sigma_{p^r,k}$ satisfies either (A), (B), or 
            \begin{itemize}
                \item[(C)] $\left( \sigma_{p^r,k}\left( mb - \overline{m} \cdot \overline{a}_l \right) \right)_{l \in I'+(c)+J'} = \left( \sigma_{p^r}\left( \overline{m} \cdot \overline{a}_z - \overline{m} \cdot \overline{a}_l \right) \right)_{l \in I'+(c)+J'}$,
            \end{itemize}
            where $z \in \left\{ i_0, j_0 \right\}$.
        \end{itemize}
    \end{clm}

    Once this claim is established, the theorem follows by the same argument as in the proof of \autoref{thm:Distality}, using \autoref{prop:RefinedQE} and \autoref{lem:RefinedAuxiliaryDistality} in place of \autoref{prop:FunctionRelativeQE} and \autoref{lem:AuxiliaryDistality}.

    \begin{proof}[Proof of Claim 1] \renewcommand{\qedsymbol}{$\blacksquare_{\mathrm{Claim}}$}
        Fix integers $m$ and $\overline{m}$.
        The proofs for (*) and (**) are the same as in \autoref{thm:Distality}, but an additional work is required to establish (***).

        If $\sigma_{p^r}\left( \overline{m} \cdot \overline{a}_{l_1} - \overline{m} \cdot \overline{a}_{i_0} \right) = \sigma_{p^r}\left( \overline{m} \cdot \overline{a}_{l_2} - \overline{m} \cdot \overline{a}_{i_0} \right) = \sigma_{p^r}\left( \overline{m} \cdot \overline{a}_{l_2} - \overline{m} \cdot \overline{a}_{l_1} \right)$ for any $l_1 < l_2$ in $I' + (c) + J'$, then since every rib has size $p$, we must have $\overline{m} \cdot \overline{a}_l - \overline{m} \cdot \overline{a}_{i_0} \in p^r G$ for any $l \in I'+(c)+J'$.
        This implies (B).

        Suppose that $\left( \sigma_{p^r}\left( \overline{m} \cdot \overline{a}_l - \overline{m} \cdot \overline{a}_{i_0} \right) \right)_{l \in I'+(c)+J'}$ is strictly increasing (the other case is analogous).
        If $\mathfrak{s}_{p^r}\left( mb - \overline{m} \cdot \overline{a}_{i_0} \right) > \left( \mathfrak{s}_{p^r}\left( \overline{m} \cdot \overline{a}_l - \overline{m} \cdot \overline{a}_{i_0} \right) \right)_{l \in I'+(c)+J'}$, then (B) holds;
        if $\mathfrak{s}_{p^r}\left( mb - \overline{m} \cdot \overline{a}_{i_0} \right) < \left( \mathfrak{s}_{p^r}\left( \overline{m} \cdot \overline{a}_l - \overline{m} \cdot \overline{a}_{i_0} \right) \right)_{l \in I'+(c)+J'}$, then (A) holds.
        Hence, we assume $\mathfrak{s}_{p^r}\left( mb - \overline{m} \cdot \overline{a}_{i_0} \right) = \left( \mathfrak{s}_{p^r}\left( \overline{m} \cdot \overline{a}_l - \overline{m} \cdot \overline{a}_{i_0} \right) \right)_{l \in I'+(c)+J'}$, which we denote by $H$, and that $G / H$ is discrete.
        If $\sigma_{p^r}\left( mb - \overline{m} \cdot \overline{a}_{i_0} - k_H \right) > \left( \sigma_{p^r}\left( \overline{m} \cdot \overline{a}_l - \overline{m} \cdot \overline{a}_{i_0} \right) \right)_{l \in I'+(c)+J'}$, then (B) follows.
        If $\sigma_{p^r}\left( mb - \overline{m} \cdot \overline{a}_{i_0} - k_H \right) < \left( \sigma_{p^r}\left( \overline{m} \cdot \overline{a}_l - \overline{m} \cdot \overline{a}_{i_0} \right) \right)_{l \in I'+(c)+J'}$, then (C) follows.
    \end{proof}
    It completes the proof.
\end{proof}

\addcontentsline{toc}{section}{References}
\begin{sloppypar}
    \printbibliography
\end{sloppypar}

\end{document}